\documentclass[12pt,onecolumn]{IEEEtran}
\usepackage[latin1]{inputenc}
\usepackage{amsmath, amssymb}
\usepackage{float}

\usepackage{amsmath,amsthm,amssymb,cancel,graphicx,url}
\usepackage[margin=2cm]{geometry}

\usepackage{fancyhdr}
\makeatletter
\usepackage{amsthm}

\usepackage{graphicx}
\usepackage{flafter}
\usepackage{chngcntr}
\let\ps@IEEEtitlepagestyle\ps@fancy
\makeatother

\usepackage{tikz}
\usetikzlibrary{matrix,arrows}

\counterwithin{figure}{section}

\newtheorem{theorem}{Theorem}[section]
\newtheorem{proposition}[theorem]{Proposition}
\newtheorem{conjecture}[theorem]{Conjecture}

\newtheorem{corollary}[theorem]{Corollary}
\newtheorem{definition}[theorem]{Definition}
\newtheorem{example}[theorem]{Example}

\newtheorem{open}[theorem]{Open question}

\begin{document}

\title{Network Communication with operators in Dedekind Finite and Stably Finite Rings}
%\author{Søren Riis}
%\affil{Centre of Discrete Mathematics}
% and School of Electronic Engineering \& Computer Science
%Queen Mary University of London
%London E1 4NS UK}

\author{\IEEEauthorblockN{Søren Riis \thanks{Email: s.riis@qmul.ac.uk} }

\IEEEauthorblockA{\normalsize Centre of Discrete Mathematics \\ School of Electronic Engineering \& Computer Science \\
Queen Mary University of London} \\  {\today}}
%\footnote{Email: s.riis@qmul.ac.uk}
%\maketitle

%
%\author{S\o ren Riis \footnote{School of Electronic Engineering and Computer Science, Queen Mary, University of London, UK. \texttt{s.riis@qmul.ac.uk}}}

%\date{\today\footnote{This work is partially supported by CNRS and Royal Society through the International Exchanges Scheme grant {\em Boolean networks, network coding and memoryless computation.}}}

%\thanks{School of Electronic Engineering and Computer Science,
%Queen Mary, University of London, London, E1 4NS, U.K. Email:
%smriis@dcs.qmul.ac.uk}

%\date{today}

\maketitle

\begin{abstract}
Messages in communication networks often are considered as "discrete" taking values in some finite alphabet (e.g. a finite field). However, if we want to consider for example communication based on analogue signals, we will have to consider messages that might be functions selected from an infinite function space. In this paper, we extend linear network coding over finite/discrete alphabets/message space to the infinite/continuous case. The key to our approach is to view the space of operators that acts linearly on a space of signals as a module over a ring.

It turns out that modules over many rings $R$ leads to unrealistic network models where communication channels have unlimited capacity.  We show that a natural condition to avoid this is equivalent to the ring $R$ being {\em Dedekind finite} (or {\em Neumann finite}) i.e. each element in $R$ has a  left inverse if and only if it has a right inverse.
We then consider a strengthened capacity condition and show that this requirement precisely corresponds to the class of (faithful) modules over {\em stably finite rings} (or {\em weakly finite}).

The introduced framework makes it possible to compare the performance of digital and analogue techniques. It turns out that within our model, digital and analogue communication outperforms each other in different situations. More specifically we construct: 1) A communications network where {\em digital communication outperforms analogue communication}
2) A communication network where {\em analogue communication outperforms digital communication}. 

The performance of a communication network is in the finite case usually measured in terms  band width (or capacity). We show this notion also remains valid for finite dimensional matrix rings which make it possible (in principle) to establish gain of digital versus analogue (analogue versus digital) communications. 

\end{abstract}
\section{Introduction - general considerations}
Control theory is a branch of engineering and mathematics that deals with the operation of dynamical
systems. 
The idea is that a controller manipulates the
inputs to obtain the desired effect of the output of the system.
The task is to ensure that one or more output variables have a particular behaviour when time progress.

Dynamic systems are often given by a graph with input and output nodes. This idea was presented in \cite{Hanne}. The graph
might be cyclic containing a complicated structure of `feedback loops'.
On the edges are attached operators that modify the signals according to certain rules and transformations.

Control theory and dynamical systems have been intensively studied for
more than 70 years  \cite{his1, his2}. Certain types of dynamic systems date back to antiquity.

Mathematical theories of information flows in networks were developed even before modern information technology was developed. Some of the ideas used in designing
Arpanet (the precursor of the internet) was for example based on queuing theory
\cite{Apranet3}.

A crucial design idea for arplanet and the internet is called "packet-switching" \cite{Apranet1, Apranet3}). The idea is that messages are broken into a discrete number of pieces (packets). These
pieces, each of which contains information on where it is supposed to go,
would then be sent out through a network and once  arriving at their destination they would be reassembled into the original message.

In modern information systems packet-switching still plays a central, crucial role. A packet of information behaves primarily as a car in a traffic system. To get from A to B, the packet of information follows a particular path. During transmission, packets are not modified or mixed
with other messages.

Contrary to this approach, in (typical) dynamic control systems signals and influences are
spread out and travels along many distinct paths and might even appear in complicated feedback structures.
Thus, we cannot realistically expect a (non-trivial) link between regular routing (packet switching) and (standard) control theory.

Recently a new approach to communication networks has been introduced. As part of multiuser information theory {\em network coding} is concerned with organising and planning information flows in communication networks. However in network coding
there is {\em no} a priori assumption that each message has to follow just one path, or that
messages cannot be modified and combined during transmission
\cite{LIN,Satellite,BR, RiisAlswede, Riis04}.
Intuitively, the behaviour of messages in network coding resembles the action of signals in dynamic systems in control theory, and in this paper I will initiate an investigation into this potential overlap of research areas.

We can consider the senders in a communication network as controllers - in the sense of control theory.  The output of the system is the messages/signals received at their destinations. This makes it possible to consider a typical communication problem as particular type of control problem.

Control theory typically deals with signals taking values in infinite "continuous" structures (e.g. complex-valued functions on some space), while messages in communication networks often are considered as "discrete" taking values in finite "discrete" alphabets/sets (e.g. finite fields).  In this paper, I will show how the finite/discrete and infinite/continuous in the linear case naturally can be unified into one general theory.

The basic idea in network coding can be explained by considering two ground stations that are communicating via a satellite (see figure \ref{fig01})

\begin{figure}[H]
\centering
\begin{tikzpicture}[descr/.style={fill=white,inner sep=2.5pt}]
  \matrix (m) [matrix of math nodes,row sep=3em,column sep=4em,minimum width=2em] {
     &  {\rm Satellite} & \\
     &  & \\
     {\rm W} & & {\rm  E} \\
      };
      \path[-stealth]
(m-3-1) edge [bend left=15] node [descr] {$x$} (m-1-2) 
(m-3-3) edge [bend right=15] node [descr] {$y$} (m-1-2)
(m-1-2) edge [bend left=15] node [descr] {$x \oplus y$} (m-3-1)
(m-1-2) edge [bend right=15] node [descr] {$x \oplus y$} (m-3-3);
\end{tikzpicture}
\caption{Satellite communication problem}
\label{fig01}
\end{figure}
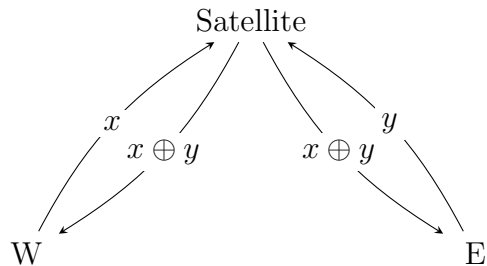

\begin{example} \label{exam01}
Suppose that {\em W} (for west) want to send a bit $x \in \{0,1\}$ to {\em E} (for east). Further, suppose that {\em E} wants to send a bit $y \in \{0,1\}$ to {\em W}.  The satellite receives the two bits $(x,y) \in \{0,1\} \times \{0,1\}$. In traditional packet routing the satellite would have to beam back to earth both the bit $x$ and the bit $y$ i.e. the satellite would have to transmit two bits. Using network coding we can reduce this to a single bit by letting the satellite transmit the exclusive OR $x \oplus y \in \{0,1\}$.  {\em W} can calculate $y=(x\oplus y) \oplus x$ and {\em E} can calculate $x=(x \oplus y) \oplus y$. \hfill $\clubsuit$
\end{example}

\bigskip

\bigskip

\noindent
We want to include to our approach the case where the space of messages is {\em infinite}. We do this by considering the space of messages as a set $M$ (commutative group) where a space of operators $R$ acts in a linear fashion. More specifically, let $\tilde{M}=(R,M)$ be a module over a ring $R$. The messages are the elements in the module (i.e. the elements in $M$) while the space of operators that acts on the signals form the ring $R$.

\begin{example} \label{exam03}
Assume that {\em W} wants to send an element $m_{{\rm W}} \in M$ to {\em E}, that simultaneously wants to send an element $m_{{\rm E}} \in M$ to {\em W}. The satellite receives the messages $m_{{\rm W}}$ and $m_{{\rm E}}$ and transmit a message $m=r_{{\rm W}} m_{{\rm W}}+
r_{{\rm E}} m_{{\rm E}}$ where $r_{{\rm W}},r_{{\rm E}} \in R$.  Assume that $r_{{\rm E}}$ has a left inverse $r_{{\rm E}}^{-1,{\rm left}}$  in $R$.
In this case {\em W} can calculate 
\[m_{{\rm E}}=r_{{\rm E}}^{-1,{\rm left}}(r_{{\rm W}} m_{{\rm W}}+r_{{\rm E}} m_{{\rm E}})-r_{{\rm E}}^{-1,{\rm left}}r_{{\rm W}}m_{{\rm W}}\]
If $r_{{\rm W}}$ has a left inverse  $r_{{\rm W}}^{-1,{\rm left}}$ station {\em E} can compute
\[m_{{\rm W}}=r_{{\rm W}}^{-1,{\rm left}}(r_{{\rm W}} m_{{\rm W}}+r_{{\rm E}} m_{{\rm E}})-r_{{\rm W}}^{-1,{\rm left}}r_{{\rm E}} m_{{\rm E}}\]
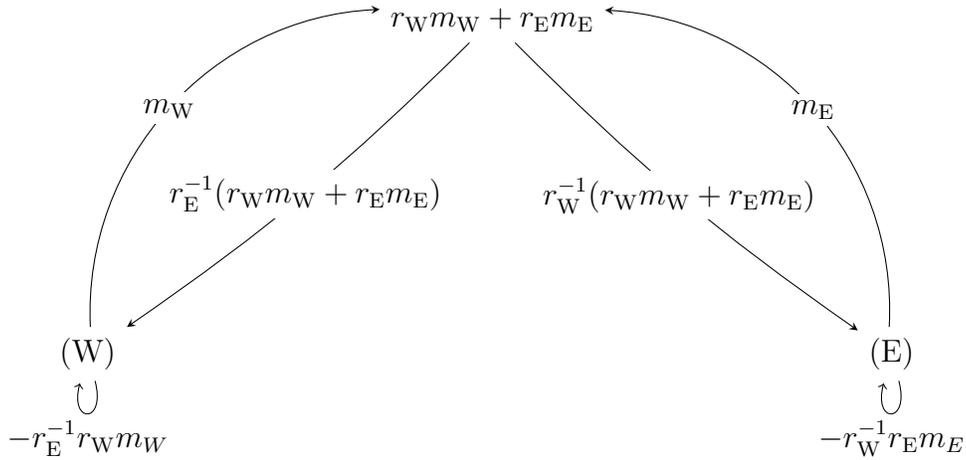
\begin{figure}[H] \label{exp}
\centering
\begin{tikzpicture}[descr/.style={fill=white,inner sep=2.5pt}]
  \matrix (m) [matrix of math nodes,row sep=3em,column sep=4em,minimum width=2em] {
     & &  r_{{\rm W}}m_{{\rm W}}+r_{{\rm E}} m_{{\rm E}}  & & \\
   & &  &  & \\
     & & & & \\
     {\rm (W)} & &  &  & {\rm (E)} \\
      };
      \path[-stealth]
(m-4-1) edge [bend left=45] node [descr] {$m_{{\rm W}}$} (m-1-3) 
(m-4-5) edge [bend right=45] node [descr] {$m_{{\rm E}}$} (m-1-3)
(m-1-3) edge [bend left=5] node [descr] {$r_{{\rm E}}^{-1}(r_{{\rm W}}m_{{\rm W}}+r_{{\rm E}}m_{{\rm E}})$} (m-4-1)
(m-1-3) edge [bend right=5] node [descr] {$r_{{\rm W}}^{-1}(r_{{\rm W}}m_{{\rm W}}+r_{{\rm E}}m_{{\rm E}})$} (m-4-5)
(m-4-1) edge [loop below] node [descr]{$-r_{{\rm E}}^{-1}r_{{\rm W}}m_W$} (m-4-1)
(m-4-5) edge [loop below] node [descr]{$-r_{{\rm W}}^{-1}r_{{\rm E}}m_E$} (m-4-5);

\end{tikzpicture}
\caption{Satellite communication problem: solution over a module $\tilde{M}=(R,M)$}
\label{fig01B}
\end{figure}

\hfill $\clubsuit$
\end{example}

\noindent
We can represent the satellite problem in figure \ref{exp} as the so-called butterfly network in figure \ref{fig02}

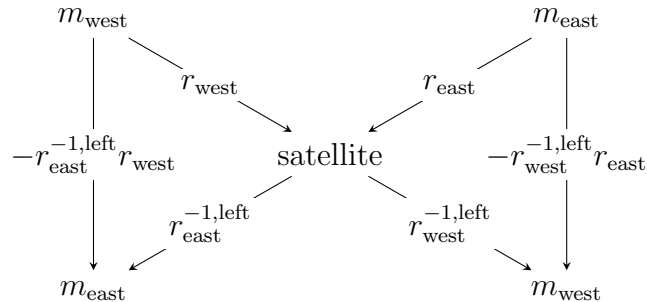
\begin{figure}[H]
\centering 

\begin{tikzpicture}[descr/.style={fill=white,inner sep=2.5pt}] 
  \matrix (m) [matrix of math nodes,row sep=3em,column sep=4em,minimum width=2em] {
     m_{{\rm west}} & &  m_{{\rm east}} \\
        &  {\rm satellite}  &  \\
     m_{{\rm  east}} &  &   m_{{\rm west}} \\};
  \path[-stealth]
(m-1-1) edge node [descr] {$r_{{\rm west}}$} (m-2-2)
(m-2-2) edge node [descr] {$r^{-1,{\rm left}}_{{\rm east}}$} (m-3-1)  
(m-1-3) edge node [descr] {$r_{{\rm east}}$} (m-2-2)
(m-2-2) edge node [descr] {$r^{-1,{\rm left}}_{{\rm west}}$}(m-3-3)
(m-1-1) edge node [descr] {$-r^{-1,{\rm left}}_{{\rm east}}r_{{\rm west}}$} (m-3-1)
(m-1-3) edge node [descr] {$-r^{-1,{\rm left}}_{{\rm west}}r_{{\rm east}}$} (m-3-3);       
\end{tikzpicture}
\caption{Butterfly network}
 \label{fig02}
\end{figure}

Notice that the sum of the products along the (two) paths from the node in the top left corner to the lower left corner is \[(r_{{\rm E}}^{-1,{\rm left}}r_{{\rm W}}- r_{{\rm E}}^{-1,{\rm left}}r_{{\rm W}})m_{{\rm W}}=0\]
and the product of the path from the upper right node to the lower left node is \[(r_{{\rm E}}^{-1,{\rm left}}r_{{\rm E}})m_{{\rm E}}=m_{{\rm E}}\]
The contribution of a node $n$ is then given by the sum of the products of all paths from the source nodes to $n$ \footnote{An elaboration of this idea that reassembles the Feynman integral can be found in \cite{reversible}}

We will use this point of calculation later when we analyse solutions to more complicated communication networks.  In the actual case, we can calculate the message received by the lower left node by summing the products along each path from the source nodes (upper left and upper right nodes) to the lower left node.Explicitly this sum of products is
\[(r_{{\rm E}}^{-1,{\rm left}}r_{{\rm W}}- r_{{\rm E}}^{-1,{\rm left}}r_{{\rm W}})m_{{\rm W}}+(r_{{\rm E}}^{-1,{\rm left}}r_{{\rm E}})m_{{\rm E}}= (0+1) m_{{\rm E}} = m_{{\rm E}} \]
which shows that the lower left node receives message $m_{{\rm E}}$ as required. A similar computation shows that the lower right node receives message $m_{{\rm W}}$ as required. 

\bigskip

\noindent
Now consider the communication network in figure \ref{fig2}

\begin{figure}[H]
\centering
\begin{tikzpicture}[descr/.style={fill=white,inner sep=2.5pt}] 
  \matrix (m) [matrix of math nodes,row sep=3em,column sep=4em,minimum width=2em] {
     m_1 & &  m_2 \\
        &  {\rm satellite}  &  \\
     m_1  &  &   m_2 \\};
  \path[-stealth]
(m-1-1) edge node [descr] {$r_1$} (m-2-2)
(m-2-2) edge node [descr] {$d_1$} (m-3-1)  
(m-1-3) edge node [descr] {$r_2$} (m-2-2)
(m-2-2) edge node [descr] {$d_2$}(m-3-3) ;       
\end{tikzpicture}
\caption{Wingless butterfly network}
\label{fig2}
\end{figure}
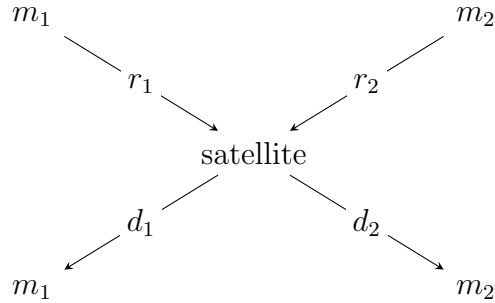

This is a network similar to the butterfly network but without the wings. The network corresponds to the case with two senders, two receivers and a satellite. Sender $1$ wants to send a message $m_1 \in M$ via the satellite to his friend receiver $1$. And  sender $2$ wants to send a message $m_2 \in M$ via the satellite to his friend receiver $2$. Can this communication be solved, so the satellite just broadcasting one message?  Intuitively, {\em in a realistic physical model it should not be possible to solve this problem as it ought to be impossible encode two messages as one}.

Now if the space of messages is infinite, it's easy to provide linear operators that solve the wingless butterfly communication problem,

\begin{example}
Let $\tilde{M}=(k,V)$ be a $k$-module where $V$ be in infinite dimension vector space over the field $k$ such that there is a $k$-isomorphism $\theta : V \rightarrow V \oplus V$.  Define projections $p_1,p_2: V \oplus V \rightarrow V$ by $p_1(v_1,v_2) :=v_1$ and $p_2(v_1,v_2):=v_2$ and define inclusions $i_1,i_2: V \rightarrow  V \oplus V$ by $i_1(v)=(v,0) \in V \oplus V$ and $i_2(v)=(0,v) \in V \oplus V$.  

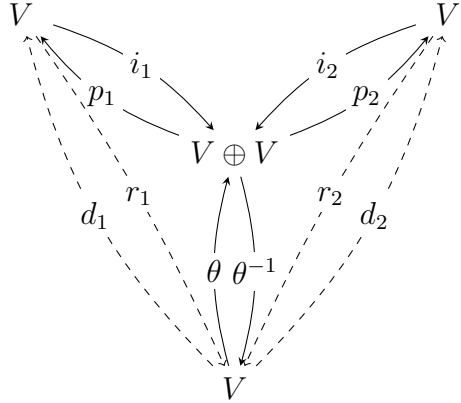
\begin{figure}[H]
\centering
\begin{tikzpicture}[descr/.style={fill=white,inner sep=2.5pt}]
  \matrix (m) [matrix of math nodes,row sep=3em,column sep=4em,minimum width=2em] {
     V & & V \\
     & V \oplus V &  \\
     & & & \\
     &       V     & \\};
  \path[-stealth]
(m-1-1) edge [bend left=15] node [descr] {$i_1$} (m-2-2)
(m-1-3) edge [bend right=15] node [descr] {$i_2$} (m-2-2)  
(m-2-2) edge [bend left=15] node [descr] {$p_1$} (m-1-1)        
(m-2-2) edge [bend right=15] node [descr] {$p_2$} (m-1-3)
(m-2-2) edge [bend left=15] node [descr] {$\theta^{-1}$}(m-4-2)  
(m-4-2) edge [bend left=15] node [descr] {$\theta$}(m-2-2);  
\path[dashed,->] 
(m-4-2) edge [bend left=15] node [descr] {$d_1$}(m-1-1)
(m-1-1) edge [bend left=5] node [descr] {$r_1$}(m-4-2)
(m-4-2) edge [bend right=15] node [descr] {$d_2$}(m-1-3)
(m-1-3) edge [bend right=5] node [descr] {$r_2$}(m-4-2);
\end{tikzpicture}
\caption{Commutative diagram}
\label{comdia}
\end{figure}

Consider the commutative diagram in figure \ref{comdia}.
The $k$-homomorphisms $r_1,r_2, d_1$ and $d_2$ on the dashed lines are uniquely determined by
$d_1:=p_1\theta, \quad d_2:=p_2 \theta, \quad r_1:=\theta^{-1}i_1, \quad r_2:=\theta^{-1}i_2$.

The communication problem in figure \ref{fig2} can now be solved by letting the satellite broadcast the message \[m_{{\rm satellite}}:=\theta^{-1}i_1v_1+\theta^{-1}i_2v_2\] where
$v_1:=m_1$, $v_2:=m_2$, and thus $m_1=p_1\theta m_{{\rm satellite}}$
and $m_2=p_2\theta m_{{\rm satellite}}$. 
\hfill $\clubsuit$
\end{example}

\bigskip

\noindent
This example illustrate that:  {\em The issue whether operators (like $r_1,r_2,d_1$ and $d_2$) should be considered realistic from a physical point of view (they are not) is determined by properties of the ring they generate.} 

The wingless butterfly communication problem can be expressed in mathematical terms as follows:  The satellite receives the messages $m_1 \in M$ and $m_2 \in M$ and broadcast a message of the form $m=r_1m_1+r_2m_2 \in M$.  For this communication to work there need to be ring elements $d_1,d_2 \in R$ that the receivers can use for decoding.  More specifically the communication problem has a solution over the module $\tilde{M}=(R,M)$ if and only 
\[(*) \quad \exists r_1,r_2 ,d_1,d_2 \in R \ \forall m_1,m_2 \in M: \quad m_1=d_1(r_1m_1+r_2m_2)  \quad {\rm and} \quad m_2=d_2(r_1m_1+r_2m_2)\]

We will often assume that $R$ acts {\em faithfully} on $M$ i.e. {\em if an operator behaves like the identity, it is the identity} i.e. $\forall r \in R: \ (\forall m \in M  \ rm=m) \rightarrow r=1$. This is logically equivalent to the the statement that {\em two operators are identical if on only if the behave the same way} i.e. $\forall r,s \in R: \  (\forall m \in M \ rm=sm) \rightarrow r=s$. This in turn is equivalent to the most common definition of $R$ acting faithfully: $\forall r \in R \setminus \{0\} \ \exists m \in M: rm \neq 0$. 

If $R$ acts faithfully on $M$ (*) is equivalent to
\[(**) \quad \exists r_1,r_2 ,d_1,d_2 \in R: \quad  d_1r_1=d_2r_2=1 \quad {\rm and} \quad d_1r_2=d_2r_1=0\]

For many rings (**) is valid,  but as we just noted in a {\em realistic physical model} it should not be possible to encode two messages as one.  From a mathematic point of view this is of course unproblematic to consider modules over rings $R$ that satisfy (**), however we want to formalise the idea that each channel has bounded finite capacity rather than unbounded capacity. We take the view that a module $\tilde{M}=(R,M)$ over a ring $R$ that satisfies (**) cannot be implemented by any realistic physical system.

\section{Dedekind finite rings and the capacity condition}

\subsection{Dedekind finite rings}

A ring $R$ (with $1$-element) is {\em Dedekind finite} (or von Neumann finite, or directly finite) if \[\forall x,y \in R : xy=1 \rightarrow yx=1\]
i.e. all one-sided inverses in R are two-sided \cite{exercisesRings}

Many classes of rings are Dedekind finite. Commutative rings, finite rings, and the matrix rings $M_n(F)$ are Dedekind finite.  Domains are Dedekind finite and so are left as well as right Noetherian rings.  For a field $k$ any finitely dimensional $k$-algebra is Dedekind finite. And for any group $G$ the group algebra $kG$ is Dedekind finite.
Rings with only finitely many nilpotent elements (i.e. elements $x$ where $x^n=0$ for some $n \in \{1,2,...\}$) are Dedekind finite.  Reversible rings i.e. rings where $\forall x,y \ (xy=0 \rightarrow yx=0)$ are Dedekind finite.  Any direct product of Dedekind finite rings is Dedekind finite.

\begin{example}
Let $V$ be a vector space and let $R$ be the ring of linear operators acting on $V$.  A linear operator $T \in R$ is injective (i.e. 1-1) if and only if it has a left-inverse, and it is surjective (i.e. onto) if and only if it has a right-inverse. More concretely, assume we try to solve an operator equation $Tx=v$ for some $v \in V$.
If $T$ has a left inverse $T^{-1,{\rm left}}$ we know that $Tx=v$ has  either no or one solution. Furthermore, if the equation has a solution, it is given by $T^{-1,{\rm left}}v$. If $T$ has a right inverse $T^{-1,{\rm right}}$ we know that $T^{-1,{\rm right}}v$ is a solution, but the equation $Tx=v$ might have other solutions.  
Dedekind finite rings are rings $R$ of operators where surjective, injective or bijective are equivalent properties. 
\hfill $\clubsuit$
\end{example}

There are various ways to express that a ring $R$ is Dedekind finite. The following list is not exhaustive but is sufficient for our purpose. 

\begin{theorem} \label{equiv}

Let $R$ be a ring with $1$ element. Then the following statements are equivalent:

\begin{enumerate}

\item \label{eq1}
$R$ is Dedekind finite i.e. $\forall x,y \in R \ ( xy=1 \rightarrow yx=1)$

\item \label{eq2}
Each element that has a right inverse has a left inverse i.e.
$\forall x \  ((\exists y \ xy=1) \rightarrow (\exists z \ zx=1))$ 

\item \label{new}
 $\forall x,y,z \in R  \ ((xy=1 \wedge xz=0) \rightarrow z=0)$ (capacity condition) 

\item  \label{eq4}
Each element has at most one right inverse i.e. $ \forall x,y,z \in R \ ((xy=1 \wedge xz=1) \rightarrow y=z)$

\item \label{eq7}
Each element that has a left inverse have a two sided inverse i.e. \\
$\forall x,y \ ((yx=1) \rightarrow (\exists z \ xz=zx=1))$

\item \label{eq3}
Each element that has a left inverse has a right inverse i.e.
$\forall x \  ((\exists y \ yx=1) \rightarrow (\exists z \ xz=1))$ 

\item \label{dual}
$\forall x,y,z \in R  \  ((yx=1 \wedge zx=0) \rightarrow z=0)$ (the dual of the capacity condition)

\item \label{eq5}
Each element has at most one left inverse i.e. $ \forall x,y,z \in R \ ((yx=1 \wedge zx=1) \rightarrow y=z)$

\item \label{eq6}
Each element that has a right inverse have a two sided inverse i.e. \\ $\forall x,y \ ((xy=1) \rightarrow (\exists z \ xz=zx=1))$

\end{enumerate}
\end{theorem}

To the best of my knowledge condition \ref{new} and condition \ref{dual} are new ways of expressing that a ring is Dedekind finite.

\noindent
{\bf Proof:} 

\noindent
(\ref{eq1}) $\implies$ (\ref{eq2}): (\ref{eq1}) can be stated as 
$\forall x \  ((\exists y \ xy=1 \rightarrow  yx=1)$ which logically implies (\ref{eq2})

\noindent
(\ref{eq2}) $\implies$ (\ref{new}): Assume $xy=1$ and $xz=0$ and assume that 
(\ref{eq2}). Then since $x$ has a right inverse $y$ it has a left inverse $w$ so $wx=1$. But then 
$0=w0=w(xz)=(wx)z=1z=z$.

\noindent
(\ref{new}) $\implies$ (\ref{eq4}):
Assume $y_1,y_2$ are right inverse of $x\in R$. Then
$xy_1=1$ and $xy_2=1$ and so $x(y_2-y_1)=0$. But then according to the 
assumption (\ref{new}) $y_2-y_1=0$ and $y_1=y_2$.

\noindent
(\ref{eq4}) $\implies$ (\ref{eq1}): Assume that $xy=1$. Consider the element $(1-yx+y)$ and notice that
$x(1-yx+y)=x-xyx+xy=xy=1$. Since we assumed that $x$ has at most one right inverse it follows that $y=(1-yx+y)$, than thus that $yx=1$.

\noindent
(\ref{eq1}) $\implies$ (\ref{eq7}): Assume that $y$ has left inverse $x$ i.e. that $xy=1$. Then according to \ref{eq1} it follows that $yx=1$ i.e. that $y$ has a two sided inverse.  

\noindent
(\ref{eq7}) $\implies$ (\ref{eq1}): Assume $xy=1$ i.e. $y$ has a left inverse. But then $y$ has a two sided inverse i.e the exist $z \in R$ s.t. $zy=yz=1$. But then $yx=yx(yz)=y(xy)z=yz=1$ i.e. $xy=yx=1$.

\bigskip

\noindent
This shows that $(\ref{eq1}) \ \Longleftrightarrow$ $(\ref{eq2}) \ \Longleftrightarrow$ $(\ref{new}) \ \Longleftrightarrow$ $(\ref{eq4}) \ \Longleftrightarrow$ (\ref{eq7}). 
The dual versions $(\ref{eq1}) \ \Longleftrightarrow$ $(\ref{eq3}) \ \Longleftrightarrow$ $(\ref{dual}) \ \Longleftrightarrow$ $(\ref{eq5}) \ \Longleftrightarrow$ (\ref{eq6}) can be shown in by simillar (but dual) arguments.

\subsection{The Capacity condition}

Consider the network in figure \ref{capacity1}. 

\begin{figure}[H]
\centering
\begin{tikzpicture}[descr/.style={fill=white,inner sep=2.5pt}]
  \matrix (m) [matrix of math nodes,row sep=3em,column sep=4em,minimum width=2em] {
     m_1 & & m_2 \\
     &  r_1m_1+r_2m_2 &  \\
     m_1 & &  \\};
  \path[-stealth]
  (m-1-1) edge node [descr] {$r_1$} (m-2-2)
  (m-1-3) edge node [descr] {$r_2=0$} (m-2-2)
  (m-2-2) edge node [descr] {$d_1$}(m-3-1);
\end{tikzpicture}
\caption{Capacity condition}
\label{capacity1}
\end{figure}
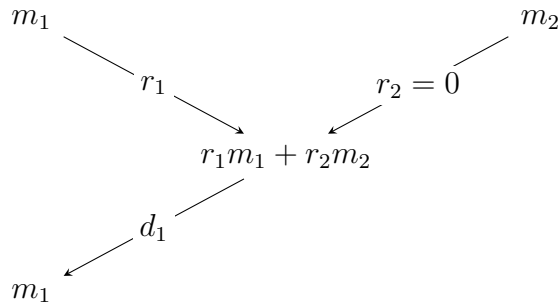

The network takes two massage $m_1$ and $m_2$ as input and the message $m_1$ can be calculated at the left lower receiver. We could have added a receiver at the lower right, but since this receiver node has no special requirement we omit it.   Intuitively the full bandwidth of the involved channels is used to transmit the message $m_1$ and thus $r_1m_1+r_2m_2$ cannot depend on $m_2$. 
In other words if $d_1$ is a decoding operator i.e. $d_1(r_1m_1+r_2m_2)=m_1$ i.e. {\em if} $d_1r_1=1$ and $d_1r_2=0$ {\em then} $r_2=0$ (see figure \ref{capacity1A} for a commutative diagram expressing the capacity condition).
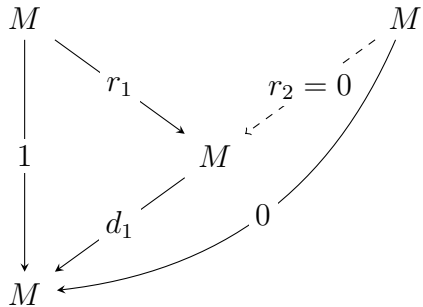
\begin{figure}[H]
\centering
\begin{tikzpicture}[descr/.style={fill=white,inner sep=2.5pt}]
  \matrix (m) [matrix of math nodes,row sep=3em,column sep=4em,minimum width=2em] {
     M & & M \\
     &  M &  \\
     M & & \\};
  \path[-stealth]
  (m-1-1) edge node [descr] {$r_1$} (m-2-2)
 
  (m-2-2) edge node [descr] {$d_1$}(m-3-1)
  (m-1-3)  edge [bend left=30] node [descr] {$0$}(m-3-1)
   (m-1-1) edge node [descr] {$1$}(m-3-1);
  \path[dashed,->] 
   (m-1-3) edge node [descr] {$r_2=0$} (m-2-2);     
                                     
\end{tikzpicture}
\caption{Commutative diagram: Capacity condition}
\label{capacity1A}
\end{figure}

This condition on $R$ can formally be expressed as: 

\begin{equation}
\forall x,y,z \in R  \ ((xy=1 \wedge xz=0) \rightarrow z=0) 
\end{equation}

We showed in the previous section that this condition is equivalent to $R$ being a Dedekind finite ring.

\begin{example}
Let  $R$ be a ring that satisfies condition (**) we considered previously i.e.
\[(**) \quad \exists r_1,r_2 ,d_1,d_2 \in R: \quad  d_1r_1=d_2r_2=1 \quad {\rm and} \quad d_1r_2=d_2r_1=0\]
Clearly $R$ is not Dedekind finite since  if $R$ were Dedekind finite $d_1r_1=1$ and $d_1r_2=0$ would imply that $r_2=0$ (by the capacity condition) which would be a contradiction since $1=d_2r_2=d_20=0$.
\hfill $\clubsuit$
\end{example}

\begin{figure}[H]
\centering
\begin{tikzpicture}[descr/.style={fill=white,inner sep=2.5pt}]
  \matrix (m) [matrix of math nodes,row sep=3em,column sep=4em,minimum width=2em] {
     m_1 & & m_2 \\
     &  r_1m_1+r_2m_2 &  \\
     m_1 & & {\rm independent \ of} \  m_1\\};
  \path[-stealth]
  (m-1-1) edge node [descr] {$r_1$} (m-2-2)
  (m-1-3) edge node [descr] {$r_2$} (m-2-2)
  (m-2-2) edge node [descr] {$d_1$}(m-3-1)
  (m-2-2) edge node [descr] {$d_2=0$}(m-3-3);
\end{tikzpicture}
\caption{Dual capacity condition}
\label{dualcapacity}
\end{figure}
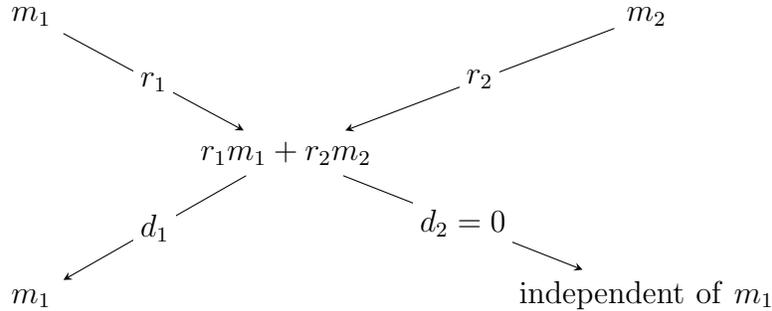

\begin{example}
Consider figure \ref{dualcapacity}. A message $m_1$ is being transmitted through a channel and being decoded such that $d_1r_1m_1=m_1$. Then intuitively $r_1m_1$ is bijective so if for some operator $d_2$, $d_2r_1m_1=0$ then $d_2=0$.  Thus if $d_1r_1=1$ and $d_2r_1=0$, then $d_2=0$ (see figure 
\ref{capacity1Adual} for a commutative diagram expressing this condition).
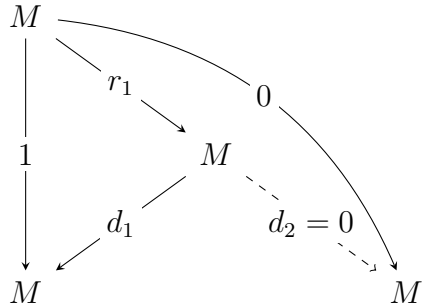
\begin{figure}[H]
\centering
\begin{tikzpicture}[descr/.style={fill=white,inner sep=2.5pt}]
  \matrix (m) [matrix of math nodes,row sep=3em,column sep=4em,minimum width=2em] {
     M & & \\
     &  M &  \\
     M & &  M\\};
  \path[-stealth]
  (m-1-1) edge node [descr] {$r_1$} (m-2-2)
  (m-2-2) edge node [descr] {$d_1$}(m-3-1)
  (m-1-1) edge [bend left=30] node [descr]{$0$}(m-3-3)
 (m-1-1) edge node [descr] {$1$}(m-3-1);
  \path[dashed,->] 
  (m-2-2) edge node [descr] {$d_2=0$}(m-3-3);                                          
\end{tikzpicture}
\caption{Commutative diagram: Dual capacity condition}
\label{capacity1Adual}
\end{figure}

This principle can be stated as: 

\begin{equation}
\forall x,y,z \in R  \  ((yx=1 \wedge zx=0) \rightarrow z=0)
\end{equation}
which is the dual capacity condition. As we showed in the previous section, this condition is also equivalent to $R$ being Dedekind finite. 
\hfill $\clubsuit$
\end{example}

\begin{example}
Polynomial identity rings (PI-rings) have been extensively investigated in the literature \cite{PIrings1,Pirings2}.  PI-rings (are isomorphic to) rings that occur by considering a commutative ring $S$ and then consider formal polynomial expressions $p \in S[x_1,x_2,...,x_n]$ in {\em non-commuting variables} that satisfies at least one identity $p(a_1,a_2,...,a_n)=0$ for all  $a_1,...,a_n \in S[x_1,....,x_n]$.  If $A$ is a PI-ring then obviously every subring of $A$ and
homomorphic image of $A$ will satisfy the identity as well. 
It can be shown that every PI ring is Dedekind finite.

The case where $p=x_1x_2-x_2x_1=0$ corresponds to commutative rings (that trivially are Dedekind finite). 

Consider the ring $A=M_2(S)$ of $2 \times 2$ matrices with coefficients in $S$. For any $a_1,a_2 \in A$, clearly ${\rm tr}(a_1a_2-a_1a_1)=0$ and thus by the Cayley-Hamilton Theorem $(a_1a_2-a_2a_1)^2=s1$ for some $s \in S$. Therefore $(a_1a_2-a_2a_1)^2$ commute with every element in $A$. Thus $A$ is a PI-ring since it satisfies the polynomial identity  $p=(x_1x_2-x_2x_1)^2x_3-x_3(x_1x_2-x_2x_1)^2=0$.  

More, generally the Amitsur-Levitzki Theorem provide a (multilinear) polynomial identity for each ring $M_k(S)$ of $k \times k$ matrices with coefficients in $S$.
 \hfill $\clubsuit$
\end{example}

\section{Finitely stable rings and the strengthened capacity condition}

\subsection{Finitely stable rings}
For a ring $R$ we can consider the matrix ring $M_k(R)$ that consists of $k \times k$ matrices with entries in $R$. Multiplication is the usual matrix multiplication and addition is matrix addition. 

\begin{definition}
We say that a ring $R$ is {\em k-stable} if the matrix ring $M_k(R)$ is Dedekind finite \footnote{There doesn't seem to be an agreed name for $k$-stable rings though this class of rings certainly have been considered in the literature}
We say  $R$ is {\em finitely stable} (or {\em fully Dedekind finite} or {\em weakly finite}) if $R$ is $k$-stable for each $k \in \{1,2,3,....\}$.
\end{definition}

\bigskip

\noindent
Commutative rings can be shown to be finitely stable. Noetherian rings and Artinian rings are stably finite. A subring of a stably finite ring and a matrix ring over a stably finite ring is stably finite.  

$1$-stable is the same as Dedekind finite. Notice that $m$-stable implies $k$-stable when $m>k$. Thus each 
$2$-stable ring is Dedekind finite.  The converse is not valid: There exists a Dedekind finite ring $R$ (in fact $R$ can be 
chosen to be a domain) such that $R$ fails to be a 2-stable. 
(see exercise 1.18 that outlines the argument). More generally there exists  for each $k>1$ a $k$-stable ring $R$ that is not 
$(k+1)$-stable.  
 
\begin{example}
The $n$th Weyl algebra is the ring $R_{{\rm weyl},n}:=k[x_1,x_2,...,x_n, \frac{\partial}{x_1},\frac{\partial}{x_2},...,\frac{\partial}{x_n}]$ of differential operators 
$\frac{\partial}{x_1},\frac{\partial}{x_2},...,\frac{\partial}{x_n}$ on the polynomial ring of $n$ variables over a field $k$ with the obvious relations. Weyl algebras are named after Hermann Weyl, who introduced them to study the
Heisenberg uncertainty principle in quantum mechanics. Each Weyl algebra $R_{{\rm weyl},n}$ is an infinitely dimensional vector space over $k$. The $n$th Weyl algebra is a simple Noetherian domain and thus stably finite.   
\hfill $\clubsuit$
\end{example}

\subsection{Requirements from finite capacity}

We will now argue that the class of Dedekind finite rings is too broad to serve as the class of rings that intuitively should be considered rings of {\em legitimate} operators.

Consider the network in figure \ref{capacity2}.

\begin{figure}[H]
\centering
\begin{tikzpicture}[descr/.style={fill=white,inner sep=2.5pt}]
  \matrix (m) [matrix of math nodes,row sep=3em,column sep=4em,minimum width=2em] {
     m_1 & & m_2 &  m_3 \\
     &  r_{14}m_1+r_{24}m_2 + r_{34}m_3 & & r_{15}m_1+r_{25}m_2+r_{35}m_3  & \\
      m_1 & & m_2 & \\};
  \path[-stealth]  
  (m-1-1) edge [bend left=0] node [descr] {$r_{14}$} (m-2-2)
  (m-1-3) edge  [bend right=0]  node [descr] {$r_{24}$} (m-2-2)
  (m-1-4) edge [bend left=10] node [descr] {$r_{34}=0$} (m-2-2)
  (m-1-1) edge [bend left=10] node [descr] {$r_{15}$} (m-2-4)
  (m-1-3) edge [bend left=20] node [descr] {$r_{25}$} (m-2-4)
  (m-1-4) edge node [descr] {$r_{35}=0$} (m-2-4)
  (m-2-2) edge node [descr] {$r_{46}$} (m-3-1)
  (m-2-2) edge node [descr] {$r_{47}$} (m-3-3)
  (m-2-4) edge [bend left=10] node [descr] {$r_{56}$}(m-3-1)
  (m-2-4) edge [bend left=10] node [descr] {$r_{57}$}(m-3-3);
\end{tikzpicture}
\caption{Strengthened capacity condition}
\label{capacity2}
\end{figure}
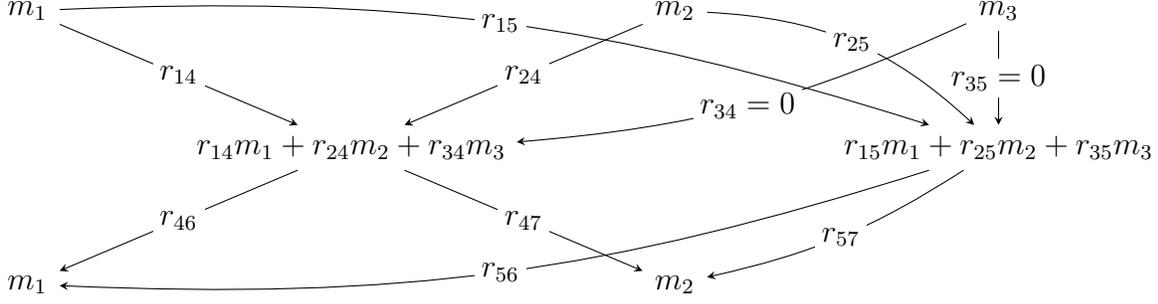

If we can assign operators $r_{ij} \in R \  i,j \in \{1,2,3,4,5,6,7,8\}$ to the edges in the graph such that $m_1$ and $m_2$ can be reconstructed correctly at the two lower  bottom nodes we get the following identitity:

\[\begin{bmatrix} r_{46} & r_{56} \\  r_{47} & r_{57}   \end{bmatrix}  \begin{bmatrix} r_{14} & r_{24} & r_{34}  \\ r_{15} & r_{25} & r_{35} \end{bmatrix} \begin{bmatrix} m_1 \\ m_2 \\ m_3 \end{bmatrix} = \begin{bmatrix} 1 & 0 & 0 \\ 0 & 1 & 0 \end{bmatrix} \begin{bmatrix} m_1 \\ m_2 \\ m_3 \end{bmatrix}= \begin{bmatrix} m_1 \\ m_2 \end{bmatrix} \]

which can also be written as
 
 \[\begin{bmatrix} r_{46} & r_{56} \\ r_{47} & r_{57} \end{bmatrix} \begin{bmatrix} r_{14} & r_{24} \\ r_{15} & r_{25} \end{bmatrix} \begin{bmatrix} m_1 \\ m_2 \end{bmatrix}=\begin{bmatrix} 1 & 0  \\ 0 & 1  \end{bmatrix} \begin{bmatrix} m_1 \\ m_2 \end{bmatrix} \]
 
 and
 \[\begin{bmatrix} r_{46} & r_{56} \\ r_{47} & r_{57} \end{bmatrix} \begin{bmatrix} r_{34} \\ r_{35} \end{bmatrix} \begin{bmatrix} m_3 \end{bmatrix}= \begin{bmatrix} 0 \\ 0 \end{bmatrix} \begin{bmatrix} m_3 \end{bmatrix}\]

 Now these matrix equations have a solution if and only if the following matrix equations are solvable for a suitable choice of operators $r_{i,j} \in R$ i.e. 
 
  \[\begin{bmatrix} r_{46} & r_{56} \\ r_{47} & r_{57} \end{bmatrix} \begin{bmatrix} r_{14} & r_{24} \\ r_{15} & r_{25} \end{bmatrix}=\begin{bmatrix} 1 & 0  \\ 0 & 1  \end{bmatrix}  \]
 
 and
 \[\begin{bmatrix} r_{46} & r_{56} \\ r_{47} & r_{57} \end{bmatrix} \begin{bmatrix} r_{34} & r'_{34} \\ r_{35}  &  r'_{35}  \end{bmatrix} = \begin{bmatrix} 0 & 0 \\ 0 & 0 \end{bmatrix} \]

This condition is equivalent to the matrix ring $M_2(R)$ satisfying the condition  

 \[\forall x,y,z \in R  \ ((xy=1 \wedge xz=0) \rightarrow z=0)\]
 
 where 
 
 \[x=\begin{bmatrix} r_{46} & r_{56} \\ r_{47} & r_{57} \end{bmatrix}, y=\begin{bmatrix} r_{14} & r_{24} \\ r_{15} & r_{25} \end{bmatrix}, z=  \begin{bmatrix} r_{34} & r'_{34} \\ r_{35}  &  r'_{35}  \end{bmatrix} \]
 
Using proposition \ref{equiv} we conclude that the matrix equations have a solution if and and only ring $M_2(R)$ Dedekind finite. Thus the matrix equations are solvable if and only if $R$ is $2$-stable.

\begin{example}  
A ring $R$ has the Invariant Basis Number property (IBN-property) if $R^m$ and $R^n$ are isomorphic as $R$ modules if and only if
$m=n$.   Stably finite rings are IBN rings.

Consider figure  \ref{capacityk} and consider the principle ${\rm Prin}(n)$ that $x_1,x_2,....x_{n+1}$ cannot all be reconstructions at the $n+1$ corresponding receiver nodes. ${\rm Prin}(1)$ is equivalent to condition $(**)$. For $n \geq 1$ the principle 
${\rm Prin}(n)$ can be shown to be equivalent to the property that $R^n$ is not isomorphic to $R^{n+1}$.  The principle ${\rm Prin}(n)$ for $n=1,2,3,...$ can be shown to be equivalent to $R$ has the IBN-property.
\end{example}
 
Consider figure \ref{capacityk}. Intuitively, if $x_1,x_2,...,x_n$ can be reconstructed correctly at the receiver notes $x_1,x_2,...,x_n$, then the message received by node $n+1$ must be independent of $x_{n+1}$.

\begin{figure}[H]
\center
\includegraphics[scale=0.5]{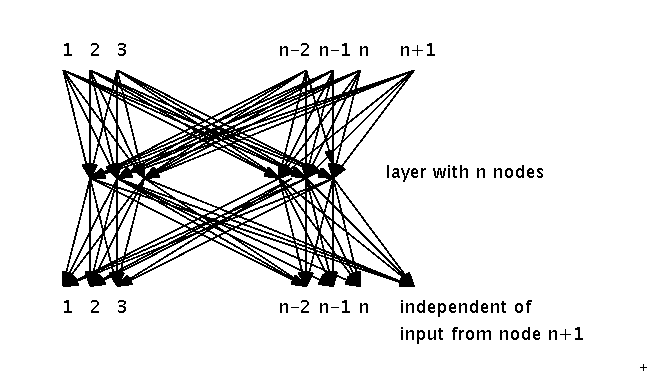} 
\caption{}
 \label{capacityk}
\end{figure}

The case $n=1$ is the network where the condition {\em Cond}$(1)$ on $R$ is equivalent to $R$ being Dedekind finite. The condition arising for $n=2$ is the condition {\em Cond}$(2)$ we just considered leading to the requirement that $M_2(R)$ is Dedekind finite.

Now consider {\em Cond}$(k)$. Let the inputs denote $m_1,m_2,...,m_k$ for nodes $1,2,...,k$ and denote the input to node $k+1$ by $m_{k+1}$.  {\em Cond}$(k)$ can be written as the condition that the matrix identity 

\[\tilde{B} \tilde{A}=1,  \tilde{B}\tilde{C}=0\]
of $k \times k$ matices in $M_k(R)$  implies that  $\tilde{C}=0$. Thus {\em Cond}$(k)$ is equivalent to the condition that the matrix ring $M_k(R)$ satisfies the condition

 \[\forall x,y,z \in R  \ ((xy=1 \wedge xz=0) \rightarrow z=0)\]
 
 Again, applying proposition \ref{equiv}  this is equivalent to the matrix ring $M_k(R)$ being Dedekind finite.  Thus we get:

\begin{proposition} \label{finitelystable}
A ring $R$ satisfies condition  {\em Cond}$(k)$ if and only if $R$ is $k$-stable.
A ring $R$ satisfies each condition  {\em Cond}$(k)$ for $k \in \{1,2,3,...\}$ if and only if $R$ is stably finite.
\end{proposition}

\section{Digital signals can be more efficient than analogue signals}

As previous consider a $R$-module $\tilde{M}=(R,M)$, where the elements in the ring $R$ acts on the elements in the additive space of messages (signals)  $M$. In general when signals (i.e. point in $M$) are real or complex valued functions, $f(t)+f(t)=0$ only when $f(t)=0$. Thus if $1 \in R$ is the identity operator  $1+1 \neq 0$.  On the other hand if data consists of binary strings ${\rm bit}(j), j \in I$ (with $I$ finite or infinite),  ${\rm bit}(j) + {\rm bit}(j)=0$ and thus $1+1=0$.  

Consider the following Information Network in figure \ref{fig4}. This network was constructed based on intuitive, informal considerations. It is possible to be more systematic and apply the constructions in \cite{poly}. However, these methods only seem to work for commutative rings, and a lot of extra work would be needed to implement this approach. Also, the resulting networks would be larger and more complex than the ad hoc network constructed here. In the communication network we have constructed in figure \ref{fig4}, the task is to transmit messages $x,y,z \in M$ from node $1,2$ and $3$ to the notes  $7,9,10$ and $11$ as indicated. 

\begin{figure}[H]
\center
 \includegraphics[scale=0.8]{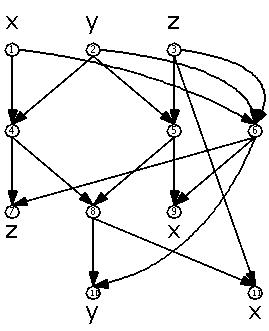} \label{fig4} \label{figure6}
 \caption{}
 \end{figure}

We will show:

\begin{theorem} \label{th4}
The Communication Network in figure \ref{fig4} has a linear solution over any ring where $1+1=0$, but has no solutions over Dedekind finite rings where $1+1 \neq 0$ 
\end{theorem}

\noindent
\begin{proof} Let us denote the operators (ring elements) assigned to the different
edges by \[r_{1,4},r_{1,6},r_{2,4},r_{2,5},r_{2,6},r_{3,5},r_{3,6},r_{3,11},r_{4,7},r_{4,8},r_{5,8},r_{5,9},r_{6,7},r_{6,9},r_{6,10},r_{8,10},r_{8,11}\]

\noindent 
The ring elements provide a solution to the communication problem if and only if the following $12$ equations hold.

\bigskip

\noindent
$(1) \quad $ paths from $i_x$ to $o_x: \quad r_{6,9}r_{1,6}=1$

\bigskip

\noindent
$(2) \quad $ paths from $i_x$ to $o_y: \quad r_{8,10}r_{4,8}r_{1,4}+r_{6,10}r_{1,6}=0$

\bigskip

\noindent
$(3) \quad $ paths from $i_x$ to $o_z: \quad r_{4,7}r_{1,4}+r_{6,7}r_{1,6}=0$

\bigskip

\noindent
$(4) \quad $ paths from $i_x$ to $\bar{o}_{x}: \quad r_{8,11}r_{4,8}r_{1,4}=1$

\bigskip

\noindent
$(5) \quad $ paths from $i_y$ to $o_x: \quad r_{5,9}r_{2,5}+r_{6,9}r_{2,6}=0$

\bigskip

\noindent
$(6) \quad $ paths from $i_y$ to $o_y: \quad r_{8,10}r_{4,8}r_{2,4}+r_{8,10}r_{5,8}r_{2,5}+r_{6,10}r_{6,2}=1$

\bigskip

\noindent
$(7) \quad $ paths from $i_y$ to $o_z:  \quad r_{4,7}r_{2,4}+r_{6,7}r_{2,6}=0$

\bigskip

\noindent
$(8) \quad $ paths from $i_y$ to $\bar{o}_{x}: \quad r_{8,11}r_{4,8}r_{2,4}+r_{8,11}r_{5,8}r_{2,5}=0$

\bigskip

\noindent
$(9) \quad $ paths from $i_z$ to $o_x: \quad r_{5,9}r_{3,5}+r_{6,9}r_{3,6}=0$

\bigskip

\noindent
$(10) \quad $ paths from $i_z$ to $o_y: \quad r_{6,10}r_{3,6}+r_{8,10}r_{5,8}r_{3,5}=0$

\bigskip

\noindent
$(11) \quad $ paths from $i_z$ to $o_z: \quad r_{6,7}r_{3,6}=1$

\bigskip

\noindent
$(12) \quad $ paths from $i_z$ to $\bar{o}_{x}: \quad r_{8,11}r_{5,8}r_{3,5}+ r_{3,11}=0$

\bigskip

\noindent
The first part of the theorem follows since the equations (1)-(12) have a solution over any ring where $1+1=0$ since equations (1)-(12) holds if we chose each $r_{ij}=1$ i.e. 
\[r_{1,4}=r_{1,6}=r_{2,4}=r_{2,5}=r_{2,6}=r_{3,5}=r_{3,6}=r_{3,11}=\] \[r_{4,7}=r_{4,8}=r_{5,8}=r_{5,9}=r_{6,7}=r_{6,9}=r_{6,10}=r_{8,10}=r_{8,11}=1\]

To show the second part of the theorem we need to show that any Dedekind finite ring $R$ that satisfies equations (1)-(12)
is forced to have $1+1=0$. 

\bigskip

\noindent
$(13)$ \quad According to (1)  $r_{6,9}r_{1,6}=1$.  Thus $r_{1,6}$ and $r_{6,9}$ each have a two-sided inverse. 

\bigskip

\noindent
$(14)$ \quad According to (4) $ r_{8,11}r_{4,8}r_{1,4}=1$.  Thus $ r_{8,11},r_{4,8}$ and $r_{1,4}$ each have a two-sided inverse.

\bigskip

\noindent
$(15)$ \quad According to (8) $r_{8,11}r_{4,8}r_{2,4}+r_{8,11}r_{5,8}r_{2,5}=0$. Since $r_{8,11}$ has a left inverse (14), it follows that
$r_{4,8}r_{2,4}+r_{5,8}r_{2,5}=0$.

\bigskip

\noindent
$(16)$  According ro (6)  \quad $r_{8,10}r_{4,8}r_{2,4}+r_{8,10}r_{5,8}r_{2,5}+r_{6,10}r_{2,6}=1$. According to (15)
$r_{8,10}(r_{4,8}r_{2,4}+r_{5,8}r_{2,5})+r_{6,10}r_{2,6}=r_{6,10}r_{2,6}=1$. Thus $r_{2,6}$ and $r_{6,10}$ each have a two-sided  inverse.

\bigskip

\noindent
$(17)$ \quad According to (5) $r_{5,9}r_{2,5}+r_{6,9}r_{2,6}=0$. We showed that $r_{2,6}$ has a two-sided inverse (16), and that $r_{6,9}$ has a two-sided inverse (13). Thus $r_{5,9}r_{2,5}r_{2,6}^{-1}r_{6,9}^{-1}=-1$ from which its straight forward to conclude that $r_{5,9}$ and  $r_{2,5}$ each have a two sided inverse.

\bigskip

\noindent
$(18)$ \quad According to (11) $r_{6,7}r_{3,6}=1$. Thus $r_{3,6}$ and $r_{6,7}$  each have a two sided inverse.

\bigskip

\noindent
$(19)$ \quad According to (9) $r_{5,9}r_{3,5}+r_{6,9}r_{3,6}=0$. We have already shown that $r_{3,6}$ has a two-sided inverse (18), and we also have shown that $r_{6,9}$ has a two-sided inverse (13). Thus $r_{5,9}r_{3,5}r_{3,6}^{-1}r_{6,9}^{-1}=-1$ and thus $r_{3,5}$ has a two sided inverse.

\bigskip

\noindent
$(20)$ \quad According to (7) $r_{4,7}r_{2,4}+r_{6,7}r_{2,6}=0$.  We already showed that $r_{6,7}$ has a two-sided inverse (18), and we also showed that $r_{2,6}$ has a two-sided inverse (16). Thus $r_{4,7}r_{2,4}r_{2,6}^{-1}r_{6,7}^{-1}=-1$.  Thus $r_{4,7}$ and $r_{2,4}$ each have a two-sided inverse.

\bigskip

\noindent
$(21)$ \quad According to (10) $r_{6,10}r_{3,6}+r_{8,10}r_{5,8}r_{3,5}=0$. Now $r_{3,6}$ and $r_{6,10}$ each has a two-sided inverse (18),(16).  From this we conclude that $r_{8,10}r_{5,8}r_{3,5}r_{3,6}^{-1}r_{6,10}^{-1}=-1$ and that $r_{8,10}$ and $r_{5,8}$ each has a two sided inverse.

\bigskip

\noindent
$(22)$ \quad According to (8) $r_{4,8}r_{2,4}+r_{5,8}r_{2,5}=0$. We already showed that $r_{2,5}$ and $r_{5,8}$ each has a two-sided inverse (17),(21).  Thus $r_{4,8}r_{2,4}r_{2,5}^{-1}r_{5,8}^{-1}=-1$ and $r_{4,8}$ has a two sided inverse.

\bigskip

\noindent
$(23)$ According to (1) \quad $r_{6,9}r_{1,6}=1$. From this we already concluded in (13) that  $r_{6,9}$ and $r_{1,6}$ each has a two-sided inverse, and thus $r_{1,6}=r_{6,9}^{-1}$.  

\bigskip

\noindent
$(24)$ \quad According to (2) $r_{8,10}r_{4,8}r_{1,4}+r_{6,10}r_{1,6}=0$ thus according to (13), (14) and (21) we can conclude that $r_{1,4}= -r_{4,8}^{-1}r_{8,10}^{-1}r_{6,10}r_{6,9}^{-1}$.

\bigskip

\noindent
$(25)$ \quad According to (3) $r_{4,7}r_{1,4}+r_{6,7}r_{1,6}=0$. Thus according to (24) it follows that
$r_{4,7}(-r_{4,8}^{-1}r_{8,10}^{-1}r_{6,10}r_{6,9}^{-1})+r_{6,7}r_{6,9}^{-1}=0$.  From this and (13) we conclude that $r_{6,7}=r_{4,7}r_{4,8}^{-1}r_{8,10}^{-1}r_{6,10}$. 

\bigskip

\noindent
$(26)$ \quad According to (4) $r_{8,11}r_{4,8}r_{1,4}=1$.  From (24) we conclude that  $r_{8,11}=r_{1,4}^{-1}r_{4,8}^{-1}=(-r_{4,8}^{-1}r_{8,10}^{-1}r_{6,10}r_{6,9}^{-1})^{-1}r_{4,8}^{-1}=-r_{6,9}r_{6,10}^{-1}r_{8,10}r_{4,8}r_{4,8}^{-1}=-r_{6,9}r_{6,10}^{-1}r_{8,10}$

\bigskip

\noindent
$(27)$ \quad According to (5) $r_{5,9}r_{2,5}+r_{6,9}r_{2,6}=0$.  According to (17) it follows that $r_{2,5}=-r_{5,9}^{-1}r_{6,9}r_{2,6}$. 

\bigskip

\noindent
$(28)$ \quad According to (8) $r_{8,11}r_{4,8}r_{2,4}+r_{8,11}r_{5,8}r_{2,5}=0$. According to (21) $r_{4,8}r_{2,4}+r_{5,8}r_{2,5}=0$ and (14) $r_{2,4}=-r_{4,8}^{-1}r_{5,8}r_{2,5}$. Substituting (22) into this we get $r_{2,4}=-r_{4,8}^{-1}r_{5,8}r_{2,5}=-r_{4,8}^{-1}r_{5,8} (-r_{5,9}^{-1}r_{6,9}r_{2,6})= r_{4,8}^{-1}r_{5,8} r_{5,9}^{-1}r_{6,9}r_{2,6}$.

\bigskip

\noindent
$(29)$ \quad According to (6) $r_{8,10}r_{4,8}r_{2,4}+r_{8,10}r_{5,8}r_{2,5}+r_{6,10}r_{2,6}=1$. According to (16) 
$r_{6,10}r_{2,6}=1$ and thus $r_{2,6}=r_{6,10}^{-1}$. 

\bigskip

\noindent
$(30)$ \quad According to (27) and (29) $r_{2,5}=-r_{5,9}^{-1}r_{6,9}r_{2,6}=-r_{5,9}^{-1}r_{6,9}r_{6,10}^{-1}$.

\bigskip

\noindent
$(31)$ \quad According to (7) $r_{4,7}r_{2,4}+r_{6,7}r_{2,6}=0$. According to (28) we have $0=r_{4,7}r_{2,4}+r_{6,7}r_{2,6}=r_{4,7} ( r_{4,8}^{-1}r_{5,8} r_{5,9}^{-1}r_{6,9}r_{6,10}^{-1})+ (r_{4,7}r_{4,8}^{-1}r_{8,10}^{-1}r_{6,10})r_{6,10}^{-1}$.
Thus $0=r_{4,7}  r_{4,8}^{-1}r_{5,8} r_{5,9}^{-1}r_{6,9}r_{6,10}^{-1}+ r_{4,7}r_{4,8}^{-1}r_{8,10}^{-1}r_{6,10}r_{6,10}^{-1}=
r_{4,7}  r_{4,8}^{-1}r_{5,8} r_{5,9}^{-1}r_{6,9}r_{6,10}^{-1}+ r_{4,7}r_{4,8}^{-1}r_{8,10}^{-1}$. Thus
$r_{8,10}=-r_{6,10}r_{6,9}^{-1}r_{5,9}r_{5,8}^{-1}$ and $r_{8,10}^{-1}=-r_{5,8}r_{5,9}^{-1}r_{6,9}r_{6,10}^{-1}$.

\bigskip

\noindent
$(32)$ \quad According to (24) and (31) $r_{1,4}=-r_{4,8}^{-1}r_{8,10}^{-1}r_{6,10}r_{6,9}^{-1}=-r_{4,8}^{-1}(-r_{5,8}r_{5,9}^{-1}r_{6,9}r_{6,10}^{-1})r_{6,10}r_{6,9}^{-1}=r_{4,8}^{-1}r_{5,8}r_{5,9}^{-1}$

\bigskip

\noindent
$(33)$ \quad According to (25) and (31)  $r_{6,7}=r_{4,7}r_{4,8}^{-1}r_{8,10}^{-1}r_{6,10}=r_{4,7}r_{4,8}^{-1}(-r_{5,8}r_{5,9}^{-1}r_{6,9}r_{6,10}^{-1})r_{6,10}=-r_{4,7}r_{4,8}^{-1}r_{5,8}r_{5,9}^{-1}r_{6,9}$.

\bigskip

\noindent
$(34)$ \quad According to (26) and (31) $r_{8,11}=-r_{6,9}r_{6,10}^{-1}r_{8,10}=-r_{6,9}r_{6,10}^{-1}(-r_{6,10}r_{6,9}^{-1}r_{5,9}r_{5,8}^{-1})=r_{5,9}r_{5,8}^{-1}$.

\bigskip

\noindent
$(35)$ According to (9) \quad $r_{5,9}r_{3,5}+r_{6,9}r_{3,6}=0$. Thus according to (17) $r_{3,5}=-r_{5,9}^{-1}r_{6,9}r_{3,6}$.

\bigskip

\noindent
$(36)$ Finally according to (10) \quad $r_{6,10}r_{3,6}+r_{8,10}r_{5,8}r_{3,5}=0$.  Combining this with (31) and (35) we get $0=r_{6,10}r_{3,6}+r_{8,10}r_{5,8}r_{3,5}=$
$r_{6,10}r_{3,6}+(-r_{6,10}r_{6,9}^{-1}r_{5,9}r_{5,8}^{-1})r_{5,8}(-r_{5,9}^{-1}r_{6,9}r_{3,6})= r_{6,10}r_{3,6}+r_{6,10}r_{3,6}=2(r_{6,10}r_{3,6})$.  

\bigskip

\noindent
$(37)$ Multiplying (36) with $r_{3,6}^{-1}r_{6,10}^{-1}$ from the right (or from the left)  we get  $0=2$.

 \bigskip
 
 \noindent
 {\em Remark:} We did not use equation (12) in this derivation. And equation (11) was only used to show each ring element $r_{ij}$ had a two sided inverse.

\end{proof}

\begin{corollary}
The network $N$ in figure \ref{fig4} cannot be solved over any ring $R$ (Dedekind finite) of operators that acts on analogue signals (i.e. $1+1 \neq 0$).  The network $N$ is solvable over any ring $R$ that acts on digital (binary) signals (i.e.where $1+1=0$).
\end{corollary}

\noindent
{\bf Remark:} The communication network is not a multiple unicast network (i.e. a communication network where each message is required at exactly one receiver node) , however it can be shown- as explained in section XII:  multiple unicast networks \cite{DNS} - that it possible to modify the network so it becomes a multiple unicast networks that separate digital from analogue. 

\section{Communication network that favours analogue signals over digital signals}

\begin{figure}[H]
\center
\includegraphics[scale=0.5]{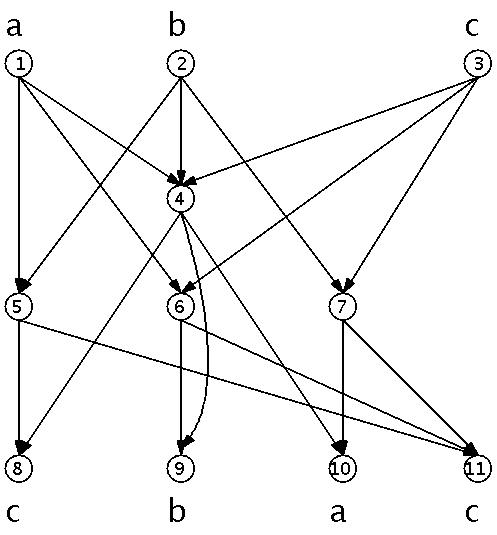} \label{analogue} 
\caption{}
\end{figure}

Consider the communication network in figure \ref{analogue}. We will show

\begin{theorem} \label{th5}
The Network in figure \ref{analogue} is solvable over any ring $R$ where the element $1+1$ is invertible. The network is {\em not} solvable over any Dedekind finite ring $R$ where $1+1$ fails to be invertible.  
\end{theorem}

\begin{proof}
Let $R$ be a Dedekind finite ring. Any solution over $R$-module (where $R$ acts faithfully)
must satisfy the following $12$ equations:

\bigskip

\noindent
(1) \quad Path from $a$ to $a$:  \quad $r_{10,4}r_{4,1}=1$

\bigskip

\noindent
(2) \quad Path from $b$ to $b$:  \quad  $r_{9,4}r_{4,2}=1$

\bigskip

\noindent
(3) \quad Path from $c$ to $c_l$: \quad $r_{8,4}r_{4,3}=1$

\bigskip

\noindent
(4) \quad Paths from $c$ to $c_r$: \quad $r_{11,7}r_{7,3}+r_{11,6}r_{6,3}=1$

\bigskip

\noindent
(5) \quad Paths from $a$ to $b$: \quad $r_{9,6}r_{6,1}+r_{9,4}r_{4,1}=0$

\bigskip

\noindent
(6) \quad Paths from $a$ to $c_l$: \quad $r_{8,5}r_{5,1}+r_{8,4}r_{4,1}=0$

\bigskip

\noindent
(7) \quad Paths from $a$ to $c_r$:  \quad $r_{11,6}r_{6,1}+r_{11,5}r_{5,1}=0$

\bigskip

\noindent
(8) \quad Paths from $b$ to $a$: \quad $r_{10,7}r_{7,2}+r_{10,4}r_{4,2}=0$

\bigskip

\noindent
(9) \quad Paths from $b$ to $c_l$: \quad $r_{8,5}r_{5,2}+r_{8,4}r_{4,2}=0$

\bigskip

\noindent
(10) \quad Paths from $b$ to $c_r$: \quad $r_{11,7}r_{7,2}+r_{11,5}r_{5,2}=0$

\bigskip

\noindent
(11) \quad Paths from $c$ to $a$: \quad $r_{10,7}r_{7,3}+r_{10,4}r_{4,3}=0$

\bigskip

\noindent
(12) \quad Paths from $c$ to $b$: \quad $r_{9,6}r_{6,3}+r_{9,4}r_{4,3}=0$

\noindent
We now use these equations and the fact that $R$ is a Dedekind finite ring so having a right inverse 
implies having an left inverse.

\bigskip

\noindent
(13) $r_{10,4}$ and $r_{4,1}$ have two sided inverses (1)

\bigskip

\noindent
(14) $r_{9,4}$ and $r_{4,2}$ have two sided inverses (2)

\bigskip

\noindent
(15)  $r_{8,4}$ and $r_{4,3}$ have two sided inverses (3)

\bigskip

\noindent
(16) $r_{9,6}r_{6,1}= - r_{9,4}r_{4,1}$ according to (5) and thus using (13) and (14) we have $r_{4,1}^{-1}r_{9,4}^{-1}(-1)r_{9,6}r_{6,1}=1$. From this we conclude that $r_{6,1}$ has a left inverse, and thus $r_{6,1}$ has a two sided inverse.

\bigskip

\noindent
(17) $r_{9,6}r_{6,1}= - r_{9,4}r_{4,1}$ (5) and thus using (13) and (14)  $r_{9,6}r_{6,1}(-1) r_{4,1}^{-1}r^{-1}_{9,4}=1$.  From this we conclude that $r_{9,6}$ has a right inverse, and thus $r_{9,6}$ has a two sided inverse.

\bigskip

\noindent
(18) We conclude that  $r_{8,5}$ has a two sided inverse. 

\bigskip

\noindent
(19) We conclude that $r_{5,1}$ has a two sided inverse.

\bigskip

\noindent
(20) $r_{10,7}$ has a two sided inverse.

\bigskip

\noindent
(21) $r_{7,2}$ has a two sided inverse.

\bigskip

\noindent
(22) $r_{9,6}$ has a two sided inverse.

\bigskip

\noindent
(23) $r_{6,3}$ has a two sided inverse.

\bigskip

\noindent
(24) According to (4), (7) and (10) and the fact $r_{6,1}$ and $r_{7,2}$ are invertible we get:
$1=(r_{11,7}r_{7,3}+r_{11,6}r_{6,3})+ (r_{11,5}r_{5,1}+r_{11,6}r_{6,1})r^{-1}_{6,1}r_{6,3} + (r_{11,5}r_{5,2}+r_{11,7}r_{7,2})r^{-1}_{7,2}r_{7,3}$

\bigskip

\noindent
(25) By expanding (24) we get:
$1=r_{11,7}r_{7,3}+r_{11,6}r_{6,3}+r_{11,6}r_{6,1}r^{-1}_{6,1}r_{6,3} + r_{11,7}r_{7,2}r^{-1}_{7,2}r_{7,3}+r_{11,5}r_{5,1}r^{-1}_{6,1}r_{6,3}+r_{11,5}r_{5,2}r^{-1}_{7,2}r_{7,3}$ 
$= 2(r_{11,7}r_{7,3}+r_{11,6}r_{6,3}) + r_{11,5}(r_{5,1}r^{-1}_{6,1}r_{6,3} + r_{5,2}r^{-1}_{7,2}r_{7,3})$

\bigskip

\noindent
(26) $r^{-1}_{6,1}= -r^{-1}_{4,1}r^{-1}_{9,4}r_{9,6}$ according to (5) and the fact that $r_{6,1},r_{4,1}$ and $r_{9,4}$ are invertible.

\bigskip

\noindent
(27) $r^{-1}_{7,2}= -r^{-1}_{4,2}r^{-1}_{10,4}r_{10,7}$ according to (8) and the fact that $r_{7,2},r_{4,2}$ and $r_{10,4}$ are invertible.

\bigskip

\noindent
(28) Substituting (26) and (27) into (25) gives:
$1= 2(r_{11,7}r_{7,3}+r_{11,6}r_{6,3}) + r_{11,5}(r_{5,1}(-r^{-1}_{4,1}r^{1}_{9,4}r_{9,6})r_{6,3} + r_{5,2} (-r^{-1}_{4,2}r^{-1}_{10,4}r_{10,7})r_{7,3}$

\bigskip

\noindent
(29) According to  (12) $r_{9,6}r_{6,3} = - r_{10,7}r_{7,3}$

\bigskip

\noindent
(30) According to  (11) $r_{10,7}r_{7,3}= - r_{10,4}r_{4,3}$

\bigskip

\noindent
(31) Inserting (29) and (30) into (28) we get:
$1=2(r_{11,7}r_{7,3}+r_{11,6}r_{6,3}) + r_{11,5}(r_{5,1}r^{-1}_{4,1}r^{-1}_{9,4}r_{9,4}r_{4,3} + 
r_{5,2} r^{-1}_{4,2}r^{-1}_{10,4}r_{10,4}r_{4,3}=2(r_{11,7}r_{7,3}+r_{11,6}r_{6,3}) + r_{11,5}(r_{5,1}r^{-1}_{4,1} + r_{5,2} r^{-1}_{4,2}) r_{4,3}$

\bigskip

\noindent
(32) $r_{5,1}= -r^{-1}_{8,5}r_{8,4}r_{4,1}$ according to (6)

\bigskip

\noindent
(33) $r_{5,2}= -r^{-1}_{8,5}r_{8,4}r_{4,2}$ according to (9)

\bigskip

\noindent
(34) Substituting (32) and (33) into (31) we get
$1=2(r_{11,7}r_{7,3}+r_{11,6}r_{6,3}) + r_{11,5}((-r^{-1}_{8,5}r_{8,4}r_{4,1}r^{-1}_{4,1}  - r^{-1}_{8,5}r_{8,4}r_{4,2} r^{-1}_{4,2}) r_{4,3}$

\bigskip

\noindent
(35) Reducing (34) we finally get:
$1=2(r_{11,7}r_{7,3}+r_{11,6}r_{6,3} - r_{11,5}r^{-1}_{8,5}r_{8,4}r_{4,3})$

\bigskip

\noindent
This shows that if the equations (1)-(12) has a solution over a Dedekind finite ring $R$, $1+1=2$ must be invertible in $R$.

On the other hand we claim that the equations (1)-(12) are solvable in any ring $R$ where $1+1$ is invertible.

\bigskip

\includegraphics[scale=0.5]{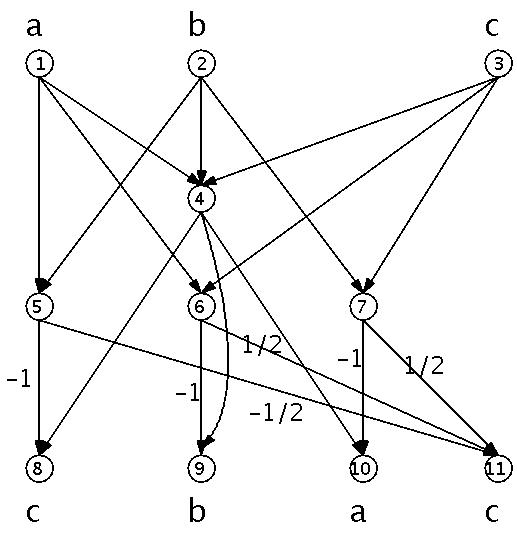} \label{analoguesol} 

\bigskip

A general solution that is valid over any ring $R$ where $2=1+1$ is invertible is indicated in figure \ref{analoguesol}. Explicitly a solution is given by letting $r_{4,1}=r_{5,1}=r_{6,1}=r_{4,2}=r_{5,2}=r_{7,2}=r_{4,3}=r_{6,3}=r_{7,3}=r_{8,4}=r_{9,4}=r_{4,10}=1$, letting $r_{8,5}=r_{9,6}=r_{10,7}=-1$ and letting $r_{11,6}=r_{11,7}= -r_{11,5}= \frac{1}{2}$ where $\frac{1}{2}$ denote the multiplicative inverse of $2$.

\end{proof}

\begin{corollary}
The network $N$ can be solved over any ring $R$ of operators that acts on analogue signals.  The network $N$ is unsolvable over any ring $R$ (that is Dedekind finite) that acts on digital (binary) signals.
\end{corollary}
\begin{proof}
Assume $M$ is a space of analogue signals (messages). We assume such signals have amplitudes (e.g. reals or complex numbers) and that we can double and half amplitudes. Thus the space $R$ of linear operators that contain the operator $2$ (doubling the amplitude) as well as its inverse  $2^{-1}$ (halfing the amplitude).  
Now, $R$ acts faithfully so $m= 2^{-1} (m+m)=2^{-1}2 m=1m$ and it follows that the doubling and the halving operators indeed are $2$ and $2^{-1}$ which satisfied $2^{-1}2=1$.

Next assume that $M$ is a space of digital (binary) signals.  More specifically, any $b \in M$ is a string (possible infinite) of binary symbols with $b+b=0$. Thus $2b=b+b=0$ for all $b \in M$, and since $R$ acts faithfully it follows that $2=1+1=0$ fail to be invertible.  
\end{proof}

\section{Analysis of Network Capacity} 

\subsection{Definitions and methods}

In the general case where $R$ is stably finite its not obvious that its possible to define the capacity (reciprocal to bandwidth) of a network in a proper manner.  To make the definition meaningful  we will strengthen the assumption on the ring $R$ from being stably finite (or Dedekind finite) and assume that $R$ is a finite dimensional matrix ring ${\it GL}_n (F)$ over some field $F$ (finite or infinite). In fact we will assume $R$ is a field, but notice that the definition of network capacity ensures that the capacity of a network with regards to a field $F$ is the same as the capacity with regards to any finite dimensional matrix ring over $F$.  With this restriction we can now define the capacity of a communication network as in \cite{capacity}.

\begin{definition}
Assume $F$ is a field and that $k,n \in \{1,2,3,....\}$. We say that a network has a solution that achieve capacity $\frac{k}{n}$ (or use bandwidth $\frac{n}{k}$) if there is a solution where each sender edge is assigned a $n \times k$ $R$-matrix (i.e. a matrix with entries in $R$),
each edge ending in a receiver node is assigned a $k \times n$ $R$-matrix, and each inner edge is assigned a $n \times n$ $R$-matrix.  The capacity of the communication network $N$ over a field $F$ is given by 
\[ {\rm capacity}(N,F) = {\rm sup}\{ \frac{k}{n} :  N {\rm  \ has \ a \ solution \ over \ } F {\rm \ that \ has \ capacity \ } 
\frac{k}{n} \}\] 
\end{definition}

\subsection{Simple example}\label{simple}
Let's first consider a rather trivial example that illustrate the idea of capacity of a communication network.

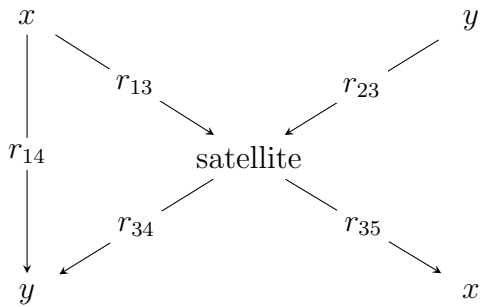
\begin{figure}[H] 
\centering
\begin{tikzpicture}[descr/.style={fill=white,inner sep=2.5pt}] 
  \matrix (m) [matrix of math nodes,row sep=3em,column sep=4em,minimum width=2em] {
     x & &  y \\
        &  {\rm satellite}  &  \\
     y &  &   x \\};
  \path[-stealth]
(m-1-1) edge node [descr] {$r_{13}$} (m-2-2)
(m-2-2) edge node [descr] {$r_{34}$} (m-3-1)  
(m-1-3) edge node [descr] {$r_{23}$} (m-2-2)
(m-1-1) edge node[descr] {$r_{14}$} (m-3-1)
(m-2-2) edge node [descr] {$r_{35}$}(m-3-3);       
\end{tikzpicture}
\caption{Simple example}
\label{figsimple}
\end{figure}

\bigskip

\noindent

\begin{proposition}
The satellite communication problem in figure \ref{figsimple} has capacity $c=\frac{1}{2}=0.5$. Equivalently, the problem can be solved if and only if the bandwidth of the communication to and from the satellite has bandwidth at least $2$.
\end{proposition}

\begin{proof}
This communication has a solution of capacity $\frac{k}{n}$ if the following system (1)-(4) of matrix equations has a solution:

\bigskip

\noindent
$(1)$ \quad $r^{k,k}_{1,4}+r^{k,n}_{3,4}r^{n,k}_{1,3}=0^{k,k}$.

\bigskip

\noindent
$(2)$ \quad $r^{k,n}_{3,5}r^{n,k}_{1,3}=1^{k,k}$.

\bigskip

\noindent
$(3)$ \quad $r^{k,n}_{3,4}r^{n,k}_{2,3}=1^{k,k}$.

\bigskip

\noindent
$(4)$ \quad $r^{k,n}_{3,5}r^{n,k}_{2,3}=0^{k,k}$.

\bigskip

\noindent
For $n=2$ and $k=1$ we have the following solution:
\[ r^{1,1}_{1,4}= \begin{bmatrix} 0  \end{bmatrix}, \quad r^{1,2}_{3,4} =  \begin{bmatrix} 0 &  1  \end{bmatrix}, \quad
r^{2,1}_{1,3}= \begin{bmatrix} 1 \\ 0 \end{bmatrix}, \quad r^{1,2}_{3,5}= \begin{bmatrix} 1 &  0  \end{bmatrix}, \quad
r^{2,1}_{2,3}=\begin{bmatrix} 0 \\ 1 \end{bmatrix}\]
This shows that  $c \geq \frac{1}{2}$.

\bigskip

\noindent
For the upper bound we proceed as follows: 
From (2) we conclude that \[k={\rm rank(1^{k,k})}={\rm rank}(r^{k,n}_{3,5}r^{n,k}_{1,3}) \leq  {\rm min}({\rm rank}(r^{k,n}_{3,5}), {\rm rank}(r^{n,k}_{1,3}))\]
Since $k \leq n$ it follows that ${\rm rank}(r^{k,n}_{3,5}) \leq k$ and ${\rm rank}(r^{n,k}_{1,3})\leq k$. From this we conclude that ${\rm rank}(r^{k,n}_{3,5})= {\rm rank}(r^{n,k}_{1,3})= k$.

\bigskip

\noindent
From (3) we conclude that \[k={\rm rank(1^{k,k})}={\rm rank}(r^{k,n}_{3,4}r^{n,k}_{2,3}) \leq  {\rm min}({\rm rank}(r^{k,n}_{3,4}), {\rm rank}(r^{n,k}_{2,3}))\]
Since $k \leq n$ it follows that ${\rm rank}(r^{k,n}_{3,4}) \leq k$ and ${\rm rank}(r^{n,k}_{2,3})\leq k$. From this we conclude that ${\rm rank}(r^{k,n}_{3,4})= {\rm rank}(r^{n,k}_{2,3})= k$.

\bigskip

\noindent
From (4) we conclude that \[2k-n={\rm rank}(r^{k,n}_{3,5})+{\rm rank}(r^{n,k}_{2,3}) - n \leq {\rm rank}(r^{k,n}_{3,5}r^{n,k}_{2,3})={\rm rank}(0^{k,k})=0\]
i.e. $n \geq 2k$ and thus that $c \leq \frac{1}{2}$. 
\end{proof}

Notice we did not use condition (1) so the network has $c=\frac{1}{2}$ with or without an edge from node $1$ to node $4$.

\section{Digital communication can achieve higher capacity  than analogue communication}

\subsection{Basic considerations (digital versus analogue)}
Let us revisit the communication network in figure \ref{figure6}.  The communication network has capacity $\frac{k}{n}$ over the field $F$ if the network has a solution where each edge from a sender mode is assigned a $n \times k$ $F$-matrix,
each edge ending in a receiver node is assigned a $k \times n$ $F$-matrix, and each inner edge is assigned a $n \times n$ $F$-matrix.  We want to lower and upper bound the capacity given by 
\[ {\rm capacity}(N,F) = {\rm sup}\{ \frac{k}{n} :  N {\rm  \ has \ a \ solution \ over \ } F {\rm \ that \ has \ capacity \ } 
\frac{k}{n} \}\] 

\begin{proposition}
The communication network in figure \ref{figure6} has capacity at least $\frac{3}{4}=0.75$, and at most $1$. 
\end{proposition}

\begin{proof}[Proof (stating the problem)] More specifically the network has a solution of capacity $\frac{k}{n}$ if the matrix equations (1)-(12) below  can be solved simultaneously. Lower and upper bounds on  the capacity of the communication network depend on the solvability of the the matrix equations (1)-(12) for various choices of $k$ and $n$. 
\bigskip

\noindent
$(1) \quad $ paths from $i_x$ to $o_x: \quad r^{k,n}_{6,9}r^{n,k}_{1,6}=1^{k,k}$

\bigskip

\noindent
$(2) \quad $ paths from $i_x$ to $o_y: \quad r^{k,n}_{8,10}r^{n,n}_{4,8}r^{n,k}_{1,4}+r^{k,n}_{6,10}r^{n,k}_{1,6}=0^{k,k}$

\bigskip

\noindent
$(3) \quad $ paths from $i_x$ to $o_z: \quad r^{k,n}_{4,7}r^{n,k}_{1,4}+r^{k,n}_{6,7}r^{n,k}_{1,6}=0^{k,k}$

\bigskip

\noindent
$(4) \quad $ paths from $i_x$ to $\bar{o}_{x}: \quad r^{k,n}_{8,11}r^{n,n}_{4,8}r^{n,k}_{1,4}=1^{k,k}$

\bigskip

\noindent
$(5) \quad $ paths from $i_y$ to $o_x: \quad r^{k,n}_{5,9}r^{n,k}_{2,5}+r^{k,n}_{6,9}r^{n,k}_{2,6}=0^{k,k}$

\bigskip

\noindent
$(6) \quad $ paths from $i_y$ to $o_y: \quad r^{k,n}_{8,10}r^{n,n}_{4,8}r^{n,k}_{2,4}+r^{k,n}_{8,10}r^{n,n}_{5,8}r^{n,k}_{2,5}+r^{k,n}_{6,10}r^{n,k}_{6,2}=1^{k,k}$

\bigskip

\noindent
$(7) \quad $ paths from $i_y$ to $o_z:  \quad r^{k,n}_{4,7}r^{n,k}_{2,4}+r^{k,n}_{6,7}r^{n,k}_{2,6}=0^{k,k}$

\bigskip

\noindent
$(8) \quad $ paths from $i_y$ to $\bar{o}_{x}: \quad r^{k,n}_{8,11}r^{n,n}_{4,8}r^{n,k}_{2,4}+r^{k,n}_{8,11}r^{n,n}_{5,8}r^{n,k}_{2,5}=0^{k,k}$

\bigskip

\noindent
$(9) \quad $ paths from $i_z$ to $o_x: \quad r^{k,n}_{5,9}r^{n,k}_{3,5}+r^{k,n}_{6,9}r^{n,k}_{3,6}=0^{k,k}$

\bigskip

\noindent
$(10) \quad $ paths from $i_z$ to $o_y: \quad r^{k,n}_{6,10}r^{n,k}_{3,6}+r^{k,n}_{8,10}r^{n,n}_{5,8}r^{n,k}_{3,5}=0^{k,k}$

\bigskip

\noindent
$(11) \quad $ paths from $i_z$ to $o_z: \quad r^{k,n}_{6,7}r^{n,k}_{3,6}=1^{k,k}$

\bigskip

\noindent
$(12) \quad $ paths from $i_z$ to $\bar{o}_{x}: \quad r^{k,n}_{8,11}r^{n,n}_{5,8}r^{n,k}_{3,5}+ r^{k,k}_{3,11}=0^{k,k}$

\bigskip

\end{proof}

\begin{proof}[Proof (lower bound)] For the lower bound we construct a solution for $n=4$ and $k=3$. The matrix solution was non-trivial to construct, but can be checked by hand.  The main idea behind the construction was to construct a solution based on "timesharing". In the solution $n$ represents $4$ time slots, and each message is split into $3$ parts.  After, extensive calculations (compute checked) we found the following matrix solution that shows that the capacity of the communication network is at least $\frac{3}{4}$:

\[r^{4,3}_{1,4}= \begin{bmatrix} 1 & 0 & 0  \\ 0 & 1 & 0  \\ 0 & 0 & 1 \\ 1 & 0 & 0  \end{bmatrix}, \quad r^{4,3}_{1,6} =  \begin{bmatrix} 1 &  0 & 0 \\ 0 & 1 & 0 \\ 0 & 0 & 1 \\ 1 & 0 & 0  \end{bmatrix}, \quad
r^{4,3}_{2,4}= \begin{bmatrix} 1 & 0 & 0 \\ 0 & 0 & 0 \\ 0 & 0 & 0 \\ 0 & 0 & 1 \end{bmatrix}, \quad r^{4,3}_{2,5}= \begin{bmatrix} -1 &  0 & 0 \\ 0 & 1 & 0 \\ 0 & 0 & 0 \\ 0 & 0 & 1  \end{bmatrix}, \quad
r^{4,3}_{2,6}=\begin{bmatrix} 1 & 0 & 0 \\ 0 & 1 & 0 \\ 0 & 0 & 0 \\ 0 & 0 & 1 \end{bmatrix}\]

\[r^{4,3}_{3,5}=\begin{bmatrix} 1 &  0 & 0 \\ 0 & 0 & 0 \\ 0 & 1 & 0 \\ 0 & 0 & 1  \end{bmatrix}, 
\quad r^{4,3}_{3,6}=\begin{bmatrix} 1 &  0 & 0 \\ 0 & 0 & 0 \\ 0 & 1 & 0 \\ 0 & 0 & 1  \end{bmatrix}, \quad
r^{3,3}_{3,11}=\begin{bmatrix} -1 &  0 & 0 \\ 0 & 0 & 0 \\ 0 & 0 & 0  \end{bmatrix}, 
\quad  r^{4,3}_{4,7}=\begin{bmatrix} -1 &  0 & 0 & 0 \\ 0 & 0 & -1 & 0 \\ 0 & 0 & 0 & -1  \end{bmatrix} \]

\[r^{4,4}_{4,8}=\begin{bmatrix} 1 &  0 & 0 & 0  \\ 0 & 1 & 0 & 0 \\ 0 & 0 & 1 & 0 \\ 0 & 0 & 0 & 1 \end{bmatrix}, \quad r^{4,4}_{5,8}=\begin{bmatrix} 1 &  0 & 0 & 0  \\ 0 & 0 & 0 & 0 \\ 0 & 0 & 0 & 0 \\ 0 & 0 & 0 & 1 \end{bmatrix}, \quad 
r^{3,4}_{5,9}=\begin{bmatrix} 0 &  0 & 0 & -1 \\ 0 & -1 & 0 & 0 \\ 0 & 0 & -1 & 0  \end{bmatrix}, \quad
r^{3,4}_{6,7}= \begin{bmatrix} 1 &  0 & 0 & 0 \\ 0 & 0 & 1 & 0 \\ 0 & 0 & 0 & 1  \end{bmatrix} \]

\[r^{3,4}_{6,9}=\begin{bmatrix} 0 &  0 & 0 & 1 \\ 0 & 1 & 0 & 0 \\ 0 & 0 & 1 & 0  \end{bmatrix}, \quad 
r^{3,4}_{6,10}=\begin{bmatrix} 1 &  0 & 0 & 0 \\ 0 & 1 & 0 & 0 \\ 0 & 0 & 0 & -1  \end{bmatrix}, \quad
r^{3,4}_{8,10}=\begin{bmatrix} -1 &  0 & 0 & 0 \\ 0 & -1 & 0 & 0 \\ 0 & 0 & 0 & 1  \end{bmatrix} \quad
r^{3,4}_{8,11}=\begin{bmatrix} 1 &  0 & 0 & 0 \\ 0 & 1 & 0 & 0 \\ 0 & 0 & 1 & 0  \end{bmatrix} \]

This solution shows that the capacity of the communication network is at least $\frac{3}{4}=0.75$. 
\end{proof}

\begin{proof}[Proof (upper bound)] The upper bound of $1$ follows from the fact that each matrix ring $M_{k}(F)$ is Dedekind finite with $1^{k,k}+1^{k,k} \neq 0^{k,k}$. Applying Theorem \ref{th4} to this fact shows that the network has no analogue solution for $k=n$. 
\end{proof}

\section{Analogue communication can achieve higher capacity  than digital communication}

\subsection{Basic considerations (Analogue versus digital)}

We have already seen that communication network in figure \ref{analogue} is only solvable over Dedekind finite rings where $1+1$ is invertible. In the section we want to quantify this difference. For the definition of capacity to be well defined we only consider the case where the underlying ring is a matrix ring $GL_{F,m}$ of $m \times m$ matrices with elements in the field $F$.  

\begin{proposition}
The communication network in figure \ref{analogue} has capacity at least $\frac{3}{4}=0.75$, and at most $1$.
\end{proposition}

\begin{proof}[Proof (stating the problem)]

\bigskip

\noindent
(1) \quad Path from $a$ to $a$:  \quad $r^{k,n}_{10,4}r^{n,k}_{4,1}=1^{k,k}$

\bigskip

\noindent
(2) \quad Path from $b$ to $b$:  \quad  $r^{k,n}_{9,4}r^{n,k}_{4,2}=1^{k,k}$

\bigskip

\noindent
(3) \quad Path from $c$ to $c_l$: \quad $r^{k,n}_{8,4}r^{n,k}_{4,3}=1^{k,k}$

\bigskip

\noindent
(4) \quad Paths from $c$ to $c_r$: \quad $r^{k,n}_{11,7}r^{n,k}_{7,3}+r^{k,n}_{11,6}r^{n,k}_{6,3}=1^{k,k}$

\bigskip

\noindent
(5) \quad Paths from $a$ to $b$: \quad $r^{k,n}_{9,6}r^{n,k}_{6,1}+r^{k,n}_{9,4}r^{n,k}_{4,1}=0^{k,k}$

\bigskip

\noindent
(6) \quad Paths from $a$ to $c_l$: \quad $r^{k,n}_{8,5}r^{n,k}_{5,1}+r^{k,n}_{8,4}r^{n,k}_{4,1}=0^{k,k}$

\bigskip

\noindent
(7) \quad Paths from $a$ to $c_r$:  \quad $r^{k,n}_{11,6}r^{n,k}_{6,1}+r^{k,n}_{11,5}r^{n,k}_{5,1}=0^{k,k}$

\bigskip

\noindent
(8) \quad Paths from $b$ to $a$: \quad $r^{k,n}_{10,7}r^{n,k}_{7,2}+r^{k,n}_{10,4}r^{n,k}_{4,2}=0^{k,k}$

\bigskip

\noindent
(9) \quad Paths from $b$ to $c_l$: \quad $r^{k,n}_{8,5}r^{n,k}_{5,2}+r^{k,n}_{8,4}r^{n,k}_{4,2}=0^{k,k}$

\bigskip

\noindent
(10) \quad Paths from $b$ to $c_r$: \quad $r^{k,n}_{11,7}r^{n,k}_{7,2}+r^{k,n}_{11,5}r^{n,k}_{5,2}=0^{k,k}$

\bigskip

\noindent
(11) \quad Paths from $c$ to $a$: \quad $r^{k,n}_{10,7}r^{n,k}_{7,3}+r^{k,n}_{10,4}r^{k,n}_{4,3}=0^{k,k}$

\bigskip

\noindent
(12) \quad Paths from $c$ to $b$: \quad $r^{k,n}_{9,6}r^{n,k}_{6,3}+r^{k,n}_{9,4}r^{n,k}_{4,3}=0^{k,k}$

\end{proof}

\begin{proof}[Proof (lower bound)]  For the lower bound we can get a solution for $n=4$ and $k=3$. Like in the solution in the previous section we construct the solution by "timesharing" where $n$ represents $4$ time slots, and each message is split into $3$ parts. Like in the previous case by extensive computer checked calculations we found the following matrix solution that shows the capacity of the communication network over rings with $1+1=0$ is at least $\frac{3}{4}$:

\[r^{4,3}_{4,1}= \begin{bmatrix} 1 & 0 & 0  \\ 0 & 1 & 0  \\ 0 & 0 & 1 \\ 0 & 0 & 0  \end{bmatrix}, \quad r^{4,3}_{5,1} =  \begin{bmatrix} 1 &  0 & 0 \\ 0 & 0 & 0 \\ 0 & 0 & 1 \\ 0 & 0 & 0  \end{bmatrix}, \quad
r^{4,3}_{6,1}= \begin{bmatrix} 1 & 0 & 0 \\ 0 & 1 & 0 \\ 0 & 0 & 1 \\ 0 & 0 & 0 \end{bmatrix}, \quad r^{4,3}_{4,2}= \begin{bmatrix} 1 &  0 & 0 \\ 0 & 1 & 0 \\ 0 & 0 & 0 \\ 0 & 0 & 1  \end{bmatrix}, \quad
r^{4,3}_{5,2}=\begin{bmatrix} 1 & 0 & 0 \\ 0 & 1 & 0 \\ 0 & 0 & 0 \\ 0 & 0 & 1 \end{bmatrix}\]

\[r^{4,3}_{7,2}=\begin{bmatrix} 1 &  0 & 0 \\ 0 & 1 & 0 \\ 0 & 0 & 0 \\ 0 & 0 & 1  \end{bmatrix}, 
\quad r^{4,3}_{4,3}=\begin{bmatrix} 1 &  0 & 0 \\ 1& 0 & 0 \\ 0 & 1 & 0 \\ 0 & 0 & 1  \end{bmatrix}, \quad
r^{4,3}_{6,3}=\begin{bmatrix} 1 &  0 & 0 \\ 1 & 0 & 0 \\ 0 & 1 & 0 \\ 0 & 0 & 1 \end{bmatrix}, 
\quad  r^{4,3}_{7,3}=\begin{bmatrix} 1 &  0 & 0 \\ 1 & 0 & 0 \\ 0 & 1 & 0 \\ 0 & 0 & 1  \end{bmatrix} \]

\[r^{3,4}_{8,4}=\begin{bmatrix} 1 &  0 & 0 & 0  \\ 0 & 0 & 1 & 0 \\ 0 & 0 & 0 & 1  \end{bmatrix}, 
\quad r^{3,4}_{10,4}=\begin{bmatrix} 1 &  0 & 0 & 0  \\ 0 & 1 & 0 & 0 \\ 0 & 0 & 1 & 0  \end{bmatrix}, 
\quad r^{3,4}_{8,5}=\begin{bmatrix} 1 &  0 & 0 & 0 \\ 0 & 0 & 1 & 0 \\ 0 & 0 & 0 & 1  \end{bmatrix}, \] 

\[r^{3,4}_{11,5}= \begin{bmatrix} 0 &  1 & 0 & 0 \\ 0 & 0 & 0 & 0 \\ 0 & 0 & 0 & 1  \end{bmatrix}  \quad
r^{3,4}_{9,4}=\begin{bmatrix} 1 &  0 & 0 & 0 \\ 0 & 1 & 0 & 0 \\ 0 & 0 & 0 & 1  \end{bmatrix}, \quad 
r^{3,4}_{9,6}=\begin{bmatrix} 1 &  0 & 0 & 0 \\ 0 & 1 & 0 & 0 \\ 0 & 0 & 0 & 1  \end{bmatrix},\]
\[r^{3,4}_{11,6}=\begin{bmatrix} 0 &  0 & 0 & 0 \\ 0 & 0 & 0 & 0 \\ 0 & 0 & 0 & 0  \end{bmatrix} \quad
r^{3,4}_{10,7}=\begin{bmatrix} 1 &  0 & 0 & 0 \\ 0 & 1 & 0 & 0 \\ 0 & 0 & 1 & 0  \end{bmatrix}  \quad
r^{3,4}_{11,7}=\begin{bmatrix} 0 &  1 & 0 & 0 \\ 0 & 0 & 1 & 0 \\ 0 & 0 & 0 & 1  \end{bmatrix} \]
 
\end{proof}

\begin{proof}[Proof (upper bound)] The upper bound of $1$ follows from the fact that each matrix ring $M_{k}(F)$ is Dedekind finite with $1^{k,k}+1^{k,k}= 0^{k,k}$. Applying Theorem \ref{th5} to this fact shows that the network has no digital solution for $k=n$. 
\end{proof}

\section{Open problems and Conclusion}
\subsection{A few specific questions}
In the simple example in section \ref{simple} we were able to provide a matching lower and upper bound. For the two main cases we considered (digital versus analogue, and analogue versus digital) there is a gap between the lower and upper bound. This naturally leads to the following questions and conjectures

\begin{open}
Determine the analogue capacity of the communication network in figure \ref{figure6}. 
\end{open}
 
 \begin{open}
Determine the digital capacity of the communication network in figure \ref{analogue}
\end{open}

\begin{conjecture}\label{conj1}
The analogue capacity of the network in figure \ref{figure6} is strictly less than $1$.
\end{conjecture} 

\begin{conjecture}\label{conj2}
The digital capacity of the network in figure \ref{analogue} is strictly less than $1$.
\end{conjecture} 

\begin{open}\label{ratio1}
Considering all communication problems $N$ where digital communication outperforms analogue communication, what is the maximal ratio  $\frac{{\rm digital \ capacity}}{{\rm analogue \ capacity}}$ 
\end{open}

\begin{open} \label{ratio2}
Considering all communication problems $N$ where analogue communication outperforms digital communication, what is the maximal ratio  $\frac{{\rm analogue \ capacity}}{{\rm digital \ capacity}}$ 
\end{open}

\begin{conjecture}
The maximal ratio in \ref{ratio1} is strictly larger than $1$
\end{conjecture} 

\begin{conjecture}
The maximal ratio in \ref{ratio2} is strictly larger than $1$
\end{conjecture} 

Notice that Conjecture \ref{conj1} implies Conjecture \ref{ratio1}, and that Conjecture \ref{conj2} implies Conjecture \ref{ratio2}.

\subsection{More open ended questions} 

Are the digital vs analogue (analogue vs digital) network in constructed in the paper the simplest possible?  

For any pair of (Dedekind finite) rings $R_1$ and $R_2$ we might ask if there exist a communication problem that is solvable over $R_1$, but unsolvable over $R_2$. And vise versa.  In certain questions of this type the general approach developed in \cite{poly} might be useful though a lot of detains will need to be checked. 

More specifically can for example the $n$th Weyl algebra be separated from the $(n+1)$th -Weyl algebra? Does there exists a communication problem that not is solvable over any commutative ring, but is solvable over a given non-commutative ring (e.g. the 1st Weyl algebra)?  Does there exist communication problems that has a solution over some Dedekind finite ring, but has no solution over any finite ring? F

Does the class of stably finite rings precisely capture the condition that bottlenecks are well behaved? Or is it possible to impose natural requirements that impose stronger restrictions than the ring being stably finite?

One general class of questions concerns {\em quantum communication}, where we consider for example a module consisting of a ring of operators acting on a Hilbert space.  It would be interesting to develop a general theory of quantum network communication based on our approach.  Does there exist communication problems that cannot be solved classically, but has quantum mechanical solutions? 
 
We defined the capacity  of a communication networks  over finitely dimension matrix rings ${\rm GL}_{n}(F)$. Can this definition be extended to a larger class of Rings? 

Finally, it would be interesting to expand the theory we developed to accommodate networks with feedback loops in the style of (non-commutative) control theory. 
 
\subsection{Conclusion}
In the paper, we developed linear network coding over rings and modules. We showed that information bottlenecks are somewhat well behaved over Dedekind finite rings, and even more so over stably finite rings.  Network coding over rings and modules seems to be a fertile area with many changeling questions that combine algebra, with problems of a graph theoretical and combinatorial flavour.

\newpage

\bibliographystyle{plain}

\bibliography{bibio}

\begin{thebibliography}{10}

\bibitem{his2}
H.S Black.
\newblock Stabilized feedback amplifiers.
\newblock {\em Bell Syst. Tech. J.}, 14:1--18, 1934.

\bibitem{poly}
R.~Dougherty and Zeger~K. Freiling, C.
\newblock Linear network codes and systems of polynomial equations.
\newblock {\em IEEE Transactions on Information Theory}, (5):2303--2316, 2008.

\bibitem{capacity}
Randall Dougherty, Chris Freiling, and Kenneth Zeger.
\newblock Unachievability of network coding capacity.
\newblock {\em IEEE Transactions on Information Theory}, 52(6):2365--2372,
  2006.

\bibitem{PIrings1}
M.~Hall.
\newblock Projective planes.
\newblock {\em Transactions Amer. Math. Soc}, pages 229--277, 1943.

\bibitem{Pirings2}
I.~Kaplansky.
\newblock Rings with a polynomial identity.
\newblock {\em Bull. Amer. Math. Soc}, (6):575--580, 1948.

\bibitem{Apranet3}
L.~Kleinrock.
\newblock {\em Stochastic Message Flow and Delay}.
\newblock McGraw-Hill Book Company, New York, 1964. (Out of Print) Reprinted by
  Dover Publications, 1972, 1964.

\bibitem{Apranet1}
L.~Kleinrock.
\newblock Sequential processing machines (spm) analyzed with a queueing theory
  model.
\newblock {\em Journal of the ACM}, 13(2), 1966.

\bibitem{BR}
R~Koetter and M~Medard.
\newblock Beyond routing: An algebraic approach to network coding.
\newblock In {\em Proceedings of the 2002 IEEE Infocom}, 2002.

\bibitem{exercisesRings}
T.Y Lam.
\newblock {\em Exercises in Modules and Rings}.
\newblock Springer-Verlag, 2007.

\bibitem{Hanne}
S.J. Mason.
\newblock Feedback theory - some properties of signal flow graphs.
\newblock {\em Proceedings of IRE}, 64:1144--1156, 1953.

\bibitem{his1}
H.~Nyquist.
\newblock Regeneration theory.
\newblock {\em Bell Syst. Tech. J.}, 11:126--147, 1932.

\bibitem{Riis04}
S.~Riis.
\newblock Linear versus non-linear boolean functions in network flow.
\newblock In {\em Proceeding of CISS 2004}.

\bibitem{DNS}
S~Riis.
\newblock Utilising public information in network coding.
\newblock In {\em General Theory of Information Transfer and Combinatorics},
  volume 4123 of {\em Lecture Notes in Computer Science}, pages 866--897, 2006.

\bibitem{reversible}
S.~Riis.
\newblock Reversible and irreversible information networks.
\newblock {\em IEEE Transactions on Information Theory}, (11):4339--4349, 2007.

\bibitem{RiisAlswede}
S.~Riis and R~Ahlswede.
\newblock Problems in network coding and error correcting codes.
\newblock NetCod 2005.

\bibitem{Satellite}
Yeung and Zhang.
\newblock Distributed source coding for satellite communications.
\newblock {\em IEEE Transactions on Information Theory}, (IT-45):1111--1120,
  1999.

\bibitem{LIN}
R~W Yeung.
\newblock Multilevel diversity coding with distortion.
\newblock {\em IEEE Transactions on Information Thory}, 41:412--422, March
  1995.

\end{thebibliography}

\end{document}